\documentclass[12pt,twoside]{amsart}

\usepackage{amsmath,amsthm,amscd,amssymb,mathrsfs,graphicx,amsfonts,mathrsfs}
\usepackage{amsfonts}
\usepackage{amssymb,enumerate}
\usepackage{amsthm}
\usepackage[all]{xy}
\usepackage{hyperref}
\allowdisplaybreaks[4] \footskip=12pt
\renewcommand{\uppercasenonmath}[1]{}

\numberwithin{equation}{section} \theoremstyle{plain}
\newtheorem*{thm*}{Main Theorem}
\newtheorem{thm}{Theorem}[section]
\newtheorem{cor}[thm]{Corollary}
\newtheorem*{cor*}{Corollary}
\newtheorem{lem}[thm]{Lemma}
\newtheorem*{lem*}{Lemma}

\newtheorem*{fact*}{Fact}

\newtheorem*{nota*}{Notation}
\newtheorem{prop}[thm]{Proposition}
\newtheorem*{prop*}{Proposition}
\newtheorem{rem}[thm]{Remark}
\newtheorem*{rem*}{Remark}

\newtheorem*{observation*}{Observation}

\newtheorem*{exa*}{Example}

\newtheorem*{df*}{Definition}

\newtheorem*{con*}{Construction}

\renewcommand{\geq}{\geqslant}
\renewcommand{\leq}{\leqslant}

\begin{document}
\begin{center}
{\large  \bf Hartshorne's question on cofinite complexes}

\vspace{0.5cm} Xiaoyan Yang and Jingwen Shen\\
Department of Mathematics, Northwest Normal University, Lanzhou 730070,
China
E-mails: yangxy@nwnu.edu.cn and shenjw0609@163.com
\end{center}

\bigskip
\centerline { \bf  Abstract}
\leftskip10truemm \rightskip10truemm \noindent Let $\mathfrak{a}$ be a proper ideal of a commutative noetherian ring $R$ and $d$ a positive integer.
We answer Hartshorne's question on cofinite complexes completely in the cases $\mathrm{dim}R=d$ or $\mathrm{dim}R/\mathfrak{a}=d-1$ or $\mathrm{ara}(\mathfrak{a})=d-1$,
show that if $d\leq2$ then an $R$-complex $X\in\mathrm{D}_\sqsubset(R)$ is $\mathfrak{a}$-cofinite if and only if each homology module $\mathrm{H}_i(X)$ is $\mathfrak{a}$-cofinite; if $\mathfrak{a}$ is a perfect ideal and $R$ is regular local with $d\leq2$ then an $R$-complex $X\in\mathrm{D}(R)$ is $\mathfrak{a}$-cofinite if and only if
$\mathrm{H}_i(X)$ is $\mathfrak{a}$-cofinite for every $i\in\mathbb{Z}$; if $d\geq3$ then for an $R$-complex $X$ of
 $\mathfrak{a}$-cofinite $R$-modules, each $\mathrm{H}_i(X)$ is $\mathfrak{a}$-cofinite if and only if $\mathrm{Ext}^j_R(R/\mathfrak{a},\mathrm{coker}d_i)$ are finitely generated for $j\leq d-2$.
We also study cofiniteness of local cohomology $\mathrm{H}^i_\mathfrak{a}(X)$ for an $R$-complex $X\in\mathrm{D}_\sqsubset(R)$ in the above cases.
The crucial step to achieve these
is to recruit the technique of spectral sequences.\\
\vbox to 0.3cm{}

\leftskip0truemm \rightskip0truemm
\bigskip
\section* { \bf Introduction}

Throughout this paper, $R$ is a commutative noetherian ring with identity. For an ideal $\mathfrak{a}$ of $R$ and an $R$-module $M$, the \emph{$i$th local cohomology} of $M$ with respect to $\mathfrak{a}$ is
\begin{center}$\mathrm{H}^i_\mathfrak{a}(M)=\underrightarrow{\textrm{lim}}_{t>0}\mathrm{Ext}^i_R(R/\mathfrak{a}^t,M)$.\end{center}
The reader can refer to \cite{BS} or \cite{Gr} for more details about local cohomology.

 Grothendieck \cite{G} asked whether $\mathrm{Hom}_R(R/\mathfrak{a},\mathrm{H}^i_\mathfrak{a}(M))$
were finitely generated for all finitely generated $R$-modules $M$ and $i\geq0$, which had
been answered affirmatively in his algebraic geometry seminars of 1961-2 when $(R,\mathfrak{m})$ is local and $\mathfrak{a}=\mathfrak{m}$.
In 1969, Hartshorne \cite{H} provided a counterexample to Grothendieck's question. He then defined an $R$-module $M$ to be \emph{$\mathfrak{a}$-cofinite}
if $\mathrm{Supp}_RM\subseteq\mathrm{V}(\mathfrak{a})$ and $\mathrm{Ext}^i_R(R/\mathfrak{a},M)$ is finitely generated for $i\geq 0$. When $R$ is an $\mathfrak{a}$-adically complete regular ring of finite Krull dimension, he
defined an $R$-complex $X$ to be \emph{$\mathfrak{a}$-cofinite} if $X\simeq\mathrm{RHom}_R(Y,\mathrm{R}\Gamma_\mathfrak{a}(R))$
for some $R$-complex $Y$ with
finitely generated homology. Hartshorne also proceeded to pose some questions in this respect.

\vspace{2mm} \noindent{\bf Question 1.}\label{Th1.4} {\it{Are the local cohomology modules $\mathrm{H}^i_\mathfrak{a}(M)$ $\mathfrak{a}$-cofinite for every finitely generated
$R$-module $M$ and every $i\geq 0$?}}

\vspace{2mm} \noindent{\bf Question 2.}\label{Th1.4} {\it{Is the category $\mathcal{M}(R,\mathfrak{a})_{cof}$ of $\mathfrak{a}$-cofinite $R$-modules is
an abelian subcategory of the category of $R$-modules?}}

\vspace{2mm} \noindent{\bf Question 3.}\label{Th1.4} {\it{Is it true that an $R$-complex $X$ is $\mathfrak{a}$-cofinite if and only if
the homology module $\mathrm{H}_i(X)$ is a-cofinite for every $i\in \mathbb{Z}$?}}
\vspace{2mm}

By providing a counterexample in \cite{H}, Hartshorne showed that the answers to these questions are
negative in general, and further established affirmative answers to Questions 1 and 2, and Question 3 for bounded above complexes when $\mathfrak{a}$ is a
principal ideal generated by a nonzerodivisor and $R$ is an $\mathfrak{a}$-adically complete regular ring with $\mathrm{dim}R<\infty$, and also when $\mathfrak{a}$ is a prime ideal of a complete
regular local ring $R$ with $\mathrm{dim}R/\mathfrak{a}=1$ (see \cite[Proposition 6.2, Corollary 6.3, Theorem 7.5, Proposition 7.6 and Corollary 7.7]{H}).

If $\mathrm{H}^i_\mathfrak{a}(M)$ is $\mathfrak{a}$-cofinite, then the set $\mathrm{Ass}_R\mathrm{H}^i_\mathfrak{a}(M)$
of associated primes
and the Bass numbers $\mu^j_{R}(\mathfrak{p},\mathrm{H}^i_\mathfrak{a}(M))$
are finite for all $\mathfrak{p}\in\mathrm{Spec}R$ and $i,j\geq0$. This observation reveals the utmost significance of Question 1 as an affirmative answer to this question in any case would readily set fourth affirmative answers
to Huneke's questions (see \cite{Hu}).
In the following years, Hartshorne's results on Questions 1 and 2 were systematically
extended and polished by commutative algebra practitioners in several stages.
Whereas not much attention
has been brought to Question 3. The most striking result on Question 3 is the following.

\vspace{2mm} \noindent{\bf Theorem I.}\label{Th1.4} {\rm (\cite[Theorem 1]{EK}, 2011).} {\it{Let $R$ be a complete Gorenstein local domain and $\mathfrak{a}$ a proper ideal of $R$ with $\mathrm{dim}R/\mathfrak{a}=1$.
Then an $R$-complex $X\in \mathrm{D}_\sqsubset(R)$ is $\mathfrak{a}$-cofinite if and only if $\mathrm{H}_i(X)$ is $\mathfrak{a}$-cofinite for every $i\in \mathbb{Z}$.}}
\vspace{2mm}

These questions are inextricably linked in an elegant way whose interrelations are yet to be
unraveled. An evidence that motivates and supports this speculation relies on the fact
that these questions stand true all together in one case and fail to hold in another case. As to shed some light on this revelation, Faridian \cite[Section 2.4]{F} defined an $R$-complex $X$ to
be \emph{$\mathfrak{a}$-cofinite} if $\mathrm{Supp}_RX\subseteq\mathrm{V}(\mathfrak{a})$ and $\mathrm{RHom}_R(R/\mathfrak{a},X)\in \mathrm{D}^\mathrm{f}(R)$, and proved that this definition coincides with that of Hartshorne.
He unearthed a connection between Hartshorne's questions, showed that if Question 3 holds true for bounded complexes then Questions 1 and 2 also hold true, and further proved the next result on Question 3.

\vspace{2mm} \noindent{\bf Theorem II.}\label{Th1.4} {\rm (\cite[Corollary 2.4.11]{F}, 2020).} {\it{Let $\mathfrak{a}$ be a proper ideal of an $\mathfrak{a}$-adically complete $R$. Suppose that either $\mathrm{cd}(\mathfrak{a},R)\leq1$, or $\mathrm{dim}R\leq2$, or $\mathrm{dim}R/\mathfrak{a}\leq 1$.
Then an $R$-complex $X\in \mathrm{D}_\square(R)$ is $\mathfrak{a}$-cofinite if and only if $\mathrm{H}_i(X)$ is $\mathfrak{a}$-cofinite for every $i\in \mathbb{Z}$.}}
\vspace{2mm}

The first primary objective of this paper is to answer Question 3 affirmatively in the cases  $\mathrm{dim}R=d\geq2$, or $\mathrm{dim}R/\mathfrak{a}=d-1$, or $\mathrm{ara}(\mathfrak{a})=d-1$ by using a spectral sequence argument. More precisely, we show that

\vspace{2mm} \noindent{\bf Theorem III.}\label{Th1.4} {\it{Let $d$ be a positive integer and $\mathfrak{a}$ a proper ideal of $R$ such that either $\mathrm{dim}R=d$, or $\mathrm{dim}R/\mathfrak{a}=d-1$, or $\mathrm{ara}(\mathfrak{a})=d-1$.

$(1)$ If $d\leq2$, then an $R$-complex $X\in \mathrm{D}_\sqsubset(R)$ is $\mathfrak{a}$-cofinite if and only if $\mathrm{H}_i(X)$ is $\mathfrak{a}$-cofinite for every $i\in \mathbb{Z}$ (see Theorem \ref{lem:1.1}).

$(2)$ If $\mathfrak{a}$ is a perfect ideal and $R$ is regular local with $d\leq2$, then an $R$-complex $X\in\mathrm{D}(R)$ is $\mathfrak{a}$-cofinite if and only if
$\mathrm{H}_i(X)$ is $\mathfrak{a}$-cofinite for every $i\in\mathbb{Z}$ (see Theorem \ref{lem:1.4}).

$(3)$ If $d\geq3$, then for a complex $X$ of
$\mathfrak{a}$-cofinite $R$-modules, each $\mathrm{H}_i(X)$ is $\mathfrak{a}$-cofinite if and only if $\mathrm{Ext}^j_R(R/\mathfrak{a},\mathrm{coker}d_i)$ is finitely generated for $j\leq d-2$ (see Theorem \ref{lem:3.43}).}}
\vspace{2mm}

We also investigate cofiniteness of local cohomology $\mathrm{H}^{i}_\mathfrak{a}(X)$ of an $R$-complex $X\in \mathrm{D}_\sqsubset(R)$ in the cases $\mathrm{dim}R=d$ or $\mathrm{dim}R/\mathfrak{a}=d-1$ or $\mathrm{ara}(\mathfrak{a})=d-1$, and find some sufficient conditions
for validity of the isomorphism $\mathrm{Ext}^{s+t}_R(R/\mathfrak{a},X)\cong\mathrm{Ext}^{s}_R(R/\mathfrak{a},\mathrm{H}^{t}_\mathfrak{a}(X))$ for any $X\in \mathrm{D}(R)$ and integers $s,t$ (see Theorem \ref{lem:3.41}).

\bigskip
\section{\bf Preliminaries}
This section collects some notions and facts of complexes
for use throughout this paper.

\vspace{2mm}
{\bf Derived category.} By an \emph{$R$-complex} $X$ we mean a sequence of $R$-modules\begin{center}$\cdots\longrightarrow
X_{n+1}\stackrel{d_{n+1}}\longrightarrow
X_n\stackrel{d_{n}}\longrightarrow
X_{n-1}\stackrel{d_{n-1}}\longrightarrow\cdots$.
\end{center}The theory of derived category is the ultimate formulation of homological algebra. The derived
category $\mathrm{D}(R)$ is defined as the localization of the homotopy category $\mathrm{K}(R)$ with respect to the
multiplicative system of quasi-isomorphisms.
 An $R$-complexes $X$ is called \emph{bounded
above} if $\mathrm{H}_n(X)=0$ for $n\gg0$, \emph{bounded below} if $\mathrm{H}_n(X)=0$ for $n\ll0$, and
\emph{bounded} if it both bounded above and bounded below. The full triangulated subcategories of $\mathrm{D}(R)$ consisting of bounded above, bounded below and bounded $R$-complexes are denoted by
$\mathrm{D}_\sqsubset(R),\mathrm{D}_\sqsupset(R)$ and $\mathrm{D}_\square(R)$. We also denote by $\mathrm{D}^\mathrm{f}(R)$ the full subcategory of $\mathrm{D}(R)$ consisting of $R$-complexes $X$ such that $\mathrm{H}_i(X)$ are
finitely generated for all $i$.
For $X\in\mathrm{D}(R)$, set
\begin{center}$\mathrm{inf}X:=\mathrm{inf}\{n\in\mathbb{Z}\hspace{0.03cm}|\hspace{0.03cm}\mathrm{H}_n(X)\neq0\},\quad \mathrm{sup}X:=\mathrm{sup}\{n\in\mathbb{Z}\hspace{0.03cm}|\hspace{0.03cm}\mathrm{H}_n(X)\neq0\}$.\end{center}

Let $X$ and $Y$ be two $R$-complexes. For every $i\in\mathbb{Z}$, let \begin{center}$\mathrm{Ext}^i_R(X,Y):=
\mathrm{H}_{-i}(\mathrm{RHom}_R(X,Y))$,\ \ $\mathrm{Tor}_i^R(X,Y):=
\mathrm{H}_{i}(X\otimes^\mathrm{L}_RY)$.\end{center}

\vspace{2mm}
{\bf Associated prime and support.} We write $\mathrm{Spec}R$ for the set of
prime ideals of $R$ and $\mathrm{Max}R$ for the set of
maximal ideals of $R$. For an ideal $\mathfrak{a}$ of $R$, we set
\begin{center}$\mathrm{V}(\mathfrak{a}):=\{\mathfrak{p}\in\textrm{Spec}R\hspace{0.03cm}|\hspace{0.03cm}\mathfrak{a}\subseteq\mathfrak{p}\}$.\end{center}
Let $M$ be an $R$-module. The set $\mathrm{Ass}_RM$ of \emph{associated prime} of $M$ is
the set of prime ideals $\mathfrak{p}$ of $R$
such that there exists a cyclic submodule $N$ of $M$ with $\mathfrak{p}=\mathrm{Ann}_RN$, the annihilator of $N$.
A prime ideal $\mathfrak{p}$ is said to be an \emph{attached prime} of $M$ if $\mathfrak{p}=\mathrm{Ann}_{R}(M/L)$ for some submodule $L$ of $M$. The set of attached primes of $M$ is denoted by $\mathrm{Att}_{R}M$. If $M$ is artinian, then $M$ admits a minimal secondary representation
$M=M_{1}+\cdots+M_{r}$ so that $M_{i}$ is $\mathfrak{p}_{i}$-secondary for $i=1,\cdots,r$. In this case, $\mathrm{Att}_{R}M=\{\mathfrak{p}_{1},\cdots,\mathfrak{p}_{r}\}$.

The \emph{support} of an $R$-module $M$ is the set \begin{center}$\mathrm{Supp}_RM:=\{\mathfrak{p}\in\textrm{Spec}R\hspace{0.03cm}|\hspace{0.03cm}M_\mathfrak{p}\neq0\}$.\end{center}For an $R$-complex $X$,
the \emph{support} of $X$ is defined as \begin{center}$\mathrm{Supp}_RX:=\{\mathfrak{p}\in\textrm{Spec}R\hspace{0.03cm}|\hspace{0.03cm}X_\mathfrak{p}\not\simeq 0\}$.\end{center}
By flatness of $R_\mathfrak{p}$ over $R$, one has $\mathrm{Supp}_RX=\bigcup_{i\in\mathbb{Z}}\mathrm{Supp}_R\mathrm{H}_i(X)$.

\vspace{2mm}
{\bf Local cohomology and local homology.}
Let $M$ be an $R$-module.

The \emph{$\mathfrak{a}$-torsion submodule} of $M$ is
\begin{center}$\Gamma_\mathfrak{a}(M):=\{m\in M\hspace{0.03cm}|\hspace{0.03cm}\mathfrak{a}^tm=0\ \textrm{for\ some\ integer}\ t\}\cong\underrightarrow{\textrm{lim}}\mathrm{Hom}_R(R/\mathfrak{a}^t,M)$.\end{center}
The module $M$ is called \emph{$\mathfrak{a}$-torsion} if $\Gamma_\mathfrak{a}(M)=M$, or equivalently, $\mathrm{Supp}_RM\subseteq\mathrm{V}(\mathfrak{a})$.
The association $M\mapsto \Gamma_\mathfrak{a}(M)$ extends to define a left exact additive functor on the
category of $R$-complexes, its right derived functor, denoted by $\mathrm{R}\Gamma_\mathfrak{a}(-)$, can be computed by $\mathrm{R}\Gamma_\mathfrak{a}(X)\simeq\Gamma_\mathfrak{a}(I)$, where $X\stackrel{\simeq}\rightarrow I$ is a semi-injective resolution of $X$. For each $R$-complex $X$ and $i\in\mathbb{Z}$, the
\emph{$i$th local cohomology} of $X$ with support in $\mathfrak{a}$ is the $R$-module \begin{center}$\mathrm{H}^i_\mathfrak{a}(X):=\mathrm{H}_{-i}(\mathrm{R}\Gamma_\mathfrak{a}(X))$.\end{center}

The \emph{$\mathfrak{a}$-adic completion} of
$M$ is \begin{center}$\Lambda^\mathfrak{a}(M):=\underleftarrow{\textrm{lim}}(R/\mathfrak{a}^t\otimes_RM)$.\end{center}
The map $M\mapsto \Lambda^\mathfrak{a}(M)$ extends to an additive functor on the category of
$R$-complexes. Following \cite{L}, this functor admits a left derived functor that is denoted by
$\mathrm{L}\Lambda^\mathfrak{a}(-)$, and that can be computed by $\mathrm{L}\Lambda^\mathfrak{a}(X)=\Lambda^\mathfrak{a}(F)$, where $F\stackrel{\simeq}\rightarrow X$ is a semi-flat resolution of $X$.
 For each $R$-complex $X$ and each integer $i$, the
\emph{$i$th derived completion} of $X$ with respect to $\mathfrak{a}$ is the $R$-module \begin{center}$\mathrm{H}_i^\mathfrak{a}(X):=\mathrm{H}_{i}(\mathrm{L}\Lambda^\mathfrak{a}(X))$.\end{center}

We need the next convergent spectral sequences, which are very useful in this paper.

\begin{lem}\label{lem:0.2}{\it{Let $\mathfrak{a},\mathfrak{b}$ be two ideals of $R$ with $\mathfrak{b}\subseteq\mathfrak{a}$. For an $R$-complex $X\in \mathrm{D}_\sqsubset(R)$, there exists a third quadrant spectral sequence
\begin{center}
$\xymatrix@C=10pt@R=5pt{
 E^2_{p,q}=\mathrm{Ext}^{-p}_R(R/\mathfrak{a},\mathrm{H}^{-q}_\mathfrak{b}(X))\ar@{=>}[r]_{\ \ \ \ \ \ p}&
 \mathrm{Ext}^{-p-q}_R(R/\mathfrak{a},X).}$\end{center}In particular, one has two spectral sequences
\begin{center}
$\xymatrix@C=10pt@R=5pt{
 E^2_{p,q}=\mathrm{Ext}^{-p}_R(R/\mathfrak{a},\mathrm{H}_{q}(X))\ar@{=>}[r]_{\ \ \ \ \ p}&
 \mathrm{Ext}^{-p-q}_R(R/\mathfrak{a},X),}$\end{center}
\begin{center}$\xymatrix@C=10pt@R=5pt{
 E^2_{p,q}=\mathrm{Ext}^{-p}_R(R/\mathfrak{a},\mathrm{H}^{-q}_\mathfrak{a}(X))\ar@{=>}[r]_{\ \ \ \ \ \ p}&
 \mathrm{Ext}^{-p-q}_R(R/\mathfrak{a},X).}$\end{center}}}
\end{lem}
\begin{proof} Set $Y:=\Sigma^{-\mathrm{sup}X}X$. By replacing $X$ with $Y$, we
may assume that $\mathrm{sup}X=0$.
Let $P$ be a projective resolution of $R/\mathfrak{a}$, and there is a semi-injective
resolution $I$ of $X$ such that $I_i=0$ for all $i>0$.
Set $M_{p,q}=\mathrm{Hom}_R(P_{-p},\Gamma_\mathfrak{b}(I_q))$. Then $\mathcal{M}=\{M_{p,q}\}$ is a third quadrant bicomplex, and hence
$\mathrm{Hom}_R(P,\Gamma_\mathfrak{b}(I))$ is the total complex of $\mathcal{M}$.
Note that
\begin{center}$\mathrm{Hom}_R(P,\Gamma_\mathfrak{b}(I))\simeq \mathrm{Hom}_R(R/\mathfrak{a},I)$\end{center}in $\mathrm{D}(R)$, we obtain the desired spectral sequence.
\end{proof}

\begin{lem}\label{lem:0.3}{\it{Let $\mathfrak{a},\mathfrak{b}$ be two ideals of $R$ with $\mathfrak{b}\subseteq\mathfrak{a}$. For an $R$-complex $X\in\mathrm{D}_\sqsupset(R)$, there exists a first quadrant spectral sequence
\begin{center}$\xymatrix@C=10pt@R=5pt{
 E^2_{p,q}=\mathrm{Tor}_{p}^R(R/\mathfrak{a},\mathrm{H}_{q}^\mathfrak{b}(X))\ar@{=>}[r]_{\ \ \ \ \ p}&
 \mathrm{Tor}_{p+q}^R(R/\mathfrak{a},X).}$\end{center}In particular, one has two spectral sequences
\begin{center}
$\xymatrix@C=10pt@R=5pt{
 E^2_{p,q}=\mathrm{Tor}_{p}^R(R/\mathfrak{a},\mathrm{H}_{q}(X))\ar@{=>}[r]_{\ \ \ \ \ p}&
 \mathrm{Tor}_{p+q}^R(R/\mathfrak{a},X),}$\end{center}
\begin{center}$\xymatrix@C=10pt@R=5pt{
 E^2_{p,q}=\mathrm{Tor}_{p}^R(R/\mathfrak{a},\mathrm{H}_{q}^\mathfrak{a}(X))\ar@{=>}[r]_{\ \ \ \ \ p}&
 \mathrm{Tor}_{p+q}^R(R/\mathfrak{a},X).}$\end{center}}}
\end{lem}
\begin{proof} Set $Y:=\Sigma^{-\mathrm{inf}X}X$. By replacing $X$ with $Y$, we
may assume that $\mathrm{inf}X=0$.
Let $F$ be a projective resolution of $R/\mathfrak{a}$, and there exists a semi-projective
resolution $P$ of $X$ such that $P_i=0$ for all $i<0$.
Set $M_{p,q}=F_p\otimes_R\Lambda^\mathfrak{b}(P_q)$. Then $\mathcal{M}=\{M_{p,q}\}$ is a first quadrant bicomplex, and so
the complex $F\otimes_R\Lambda^\mathfrak{b}(P)$ is the total complex of $\mathcal{M}$. Since
\begin{center}$F\otimes_R\Lambda^\mathfrak{b}(P)\simeq R/\mathfrak{a}\otimes_RP$\end{center}in $\mathrm{D}(R)$,
 we obtain the desired spectral sequence.
\end{proof}

\bigskip
\section{\bf Cofiniteness of complexes in the case $d\leq2$}
The task of this section is to answer Question 2 completely in the cases $\mathrm{dim}R/\mathfrak{a}\leq1$, or $\mathrm{dim}R\leq2$, or $\mathrm{ara}(\mathfrak{a})\leq1$.

Recall that a \emph{Serre subcategory} $\mathcal{S}$ of the category of $R$-modules is a class such
that for any short exact sequence
$0\rightarrow L\rightarrow M\rightarrow N\rightarrow 0$, $M$ is in $\mathcal{S}$ if and only if $L$ and $N$ are in $\mathcal{S}$.
The next result is frequently used through the article, enable us to demonstrate some new facts and improve some older facts about
the local cohomology.

\begin{lem}\label{lem:0.6}{\it{Let $\mathcal{S}$ be a Serre subcategory of $R$-modules and $\mathfrak{a},\mathfrak{b}$ two ideals of $R$ with $\mathfrak{b}\subseteq\mathfrak{a}$, and let $X\in\mathrm{D}_\sqsubset(R)$ and $s\geq0,t\geq-\mathrm{sup}X$ such that

$(1)$ $\mathrm{Ext}^{s+t}_R(R/\mathfrak{a},X)$ is in $\mathcal{S}$;

$(2)$ $\mathrm{Ext}^{s+1+i}_R(R/\mathfrak{a},\mathrm{H}^{t-i}_\mathfrak{b}(X))$ is in $\mathcal{S}$ for all $1\leq i\leq t+\mathrm{sup}X$;

$(3)$ $\mathrm{Ext}^{s-1-i}_R(R/\mathfrak{a},\mathrm{H}^{t+i}_\mathfrak{b}(X))$ is in $\mathcal{S}$ for all $1\leq i\leq s-1$.\\
Then the $R$-module $\mathrm{Ext}^{s}_R(R/\mathfrak{a},\mathrm{H}^{t}_\mathfrak{b}(X))$ belongs to $\mathcal{S}$.}}
\end{lem}
\begin{proof} Consider the first spectral sequence in Lemma \ref{lem:0.2}. For $s=0$, there is a finite filtration\begin{center}
$0=U^{-t-\mathrm{sup}X-1}\subseteq U^{-t-\mathrm{sup}X}\subseteq\cdots \subseteq U^{0}=\mathrm{Ext}^{t}_R(R/\mathfrak{a},X)$,
\end{center}such that $U^{p}/U^{p-1}\cong E^\infty_{p,-t-p}$ for $p+\mathrm{sup}X\geq -t$. Let $r\geq 2$. Consider the differential
\begin{center}$0=E^r_{r,-t-r+1}\xrightarrow{d^r_{r,-t-r+1}}E^r_{0,-t}
\xrightarrow{d^r_{0,-t}}E^r_{-r,-t+r-1}.$
\end{center}We obtain
 the following short exact sequence
\begin{center}$0\rightarrow E^{r+1}_{0,-t}\rightarrow E^{r}_{0,-t}\rightarrow\mathrm{im}d^r_{0,-t}\rightarrow0$.
\end{center}
Since $-t+r-1\geq -t+1$ and $E^r_{-r,-t+r-1}$
is a subquotient of $E^2_{-r,-t+r-1}$, it follows that $\mathrm{im}d^r_{0,-t}\in\mathcal{S}$ for $r\geq 2$. Since $\mathrm{sup}X$ is finite, $E^{r}_{0,-t}\cong E^{\infty}_{0,-t}\cong U^{0}/U^{-1}\in\mathcal{S}$ for $r\gg0$.
By using the above sequence
inductively, one has
$\mathrm{Hom}_R(R/\mathfrak{a},\mathrm{H}^{t}_\mathfrak{b}(X))\cong E^{2}_{0,-t}\in\mathcal{S}$. For $s=1$, there exists a finite filtration\begin{center}
$0=U^{-t-\mathrm{sup}X-2}\subseteq U^{-t-\mathrm{sup}X-1}\subseteq\cdots \subseteq U^{0}=\mathrm{Ext}^{t+1}_R(R/\mathfrak{a},X)$,
\end{center}so that $U^{p}/U^{p-1}\cong E^\infty_{p,-t-1-p}$ for $p+\mathrm{sup}X\geq -t-1$. Let $r\geq 2$. Consider the differential
\begin{center}$0=E^r_{-1+r,-t-r+1}\xrightarrow{d^r_{-1+r,-t-r+1}}E^r_{-1,-t}
\xrightarrow{d^r_{-1,-t}}E^r_{-1-r,-t+r-1}.$
\end{center}
As $-t+r-1\geq -t+1$, it follows that $\mathrm{im}d^r_{-1,-t}\in\mathcal{S}$ for $r\geq 2$. Since
$E^{r}_{-1,-t}\cong E^{\infty}_{-1,-t}\cong U^{-1}/U^{-2}\in\mathcal{S}$ for $r\gg0$,
 the short exact sequence
\begin{center}$0\rightarrow E^{r+1}_{-1,-t}\rightarrow E^{r}_{-1,-t}\rightarrow\mathrm{im}d^r_{-1,-t}\rightarrow0$
\end{center} implies that
$\mathrm{Ext}^1_R(R/\mathfrak{a},\mathrm{H}^{t}_\mathfrak{b}(X))\cong E^{2}_{-1,-t}\in\mathcal{S}$.
For $s\geq 2$,
consider the filtration\begin{center}
$0=U^{-s-t-\mathrm{sup}X-1}\subseteq U^{-s-t-\mathrm{sup}X}\subseteq\cdots \subseteq U^{0}=\mathrm{Ext}^{s+t}_R(R/\mathfrak{a},X)$,
\end{center}where $U^{p}/U^{p-1}\cong E^\infty_{p,-s-t-p}$ for $-s-t\leq p+\mathrm{sup}X$. Let $r\geq 2$. Consider the differential
\begin{center}$E^r_{-s+r,-t-r+1}\xrightarrow{d^r_{-s+r,-t-r+1}}E^r_{-s,-t}
\xrightarrow{d^r_{-s,-t}}E^r_{-s-r,-t+r-1}.$
\end{center}
As $E^r_{-s-r,-t+r-1}=0$ for $r\geq t+\mathrm{sup}X+2$ and $E^r_{-s+r,-t-r+1}=0$ for $r\geq s+1$, it follows from the conditions that $\mathrm{im}d^r_{-s+r,-t-r+1}$ and $\mathrm{im}d^r_{-s,-t}$ are in $\mathcal{S}$ for $r\geq 2$.
Let $r\geq s+1$. Then $E^r_{-s+r,-t-r+1}=0$. So we have a short exact sequence \begin{center}$
 0\rightarrow E^{r+1}_{-s,-t}\rightarrow E^{r}_{-s,-t}\rightarrow \mathrm{im}d^{r}_{-s,-t}\rightarrow0$.
\end{center}
Since $E^{r}_{-s,-t}\cong E^{\infty}_{-s,-t}\cong U^{-s}/U^{-s-1}\in\mathcal{S}$ for $r\gg0$, it follows from the above sequence that
$E^{s+1}_{-s,-t}\in\mathcal{S}$. Thus the following exact sequence
\begin{center}$
 0\rightarrow \mathrm{im}d^{s}_{0,-t-s+1}\rightarrow \mathrm{ker}d^{s}_{-s,-t}\rightarrow E^{s+1}_{-s,-t}\rightarrow0$,
\end{center}implies that $\mathrm{ker}d^{s}_{-s,-t}\in\mathcal{S}$, and the next exact sequence \begin{center}$
 0\rightarrow \mathrm{ker}d^{s}_{-s,-t}\rightarrow E^{s}_{-s,-t}\rightarrow \mathrm{im}d^{s}_{-s,-t}\rightarrow0$,
\end{center}yields that $E^{s}_{-s,-t}\in\mathcal{S}$.
By repeating this process, we obtain that
$\mathrm{Ext}^s_R(R/\mathfrak{a},\mathrm{H}^{t}_\mathfrak{b}(X))\cong E^{2}_{-s,-t}\in\mathcal{S}$. The proof is complete.
\end{proof}

\begin{cor}\label{lem:1.8}{\it{Let $\mathcal{S}$ be a Serre subcategory of $R$-modules and $\mathfrak{a},\mathfrak{b}$ two ideals of $R$ with $\mathfrak{b}\subseteq\mathfrak{a}$, and let $X\in\mathrm{D}_\sqsubset(R)$ and $s\geq0,t\geq-\mathrm{sup}X$. If $\mathrm{Ext}^s_R(R/\mathfrak{a},X)\in\mathcal{S}$ and $\mathrm{Ext}^s_R(R/\mathfrak{a},\mathrm{H}^{j}_\mathfrak{b}(X))\in\mathcal{S}$ for all $s$ and $j\neq t$, then $\mathrm{Ext}^s_R(R/\mathfrak{a},\mathrm{H}^{t}_\mathfrak{b}(X))\in\mathcal{S}$ for all $s$.}}
\end{cor}

Let $n\geq-1$ be an integer. Recall that an $R$-module $M$ is said to be in $FD_{\leq n}$
if there is a finitely generated submodule $N$ of $M$ such that $\mathrm{dim}_RM/N\leq n$.
By definition, any finitely generated $R$-module and any $R$-module with dimension
at most $n$ are in $FD_{\leq n}$. An $R$-module $M$ is said to be \emph{weakly Laskerian} if the set $\mathrm{Ass}_RM/N$
is finite for each submodule $N$ of $M$. An $R$-module $M$ is \emph{minimax} if there is a
finitely generated submodule $N$ of $M$, such that $M/N$ is artinian. If $M$ is minimax then
$M\in FD_{\leq 0}$; if $M$ is weakly Laskerian then $M\in FD_{\leq 1}$ (see \cite[Remark 2.2]{AN}).

\begin{lem}\label{lem:6.6}{\it{Let $n=0,1$ and $M$ be in $FD_{\leq n}$. Then
 $M$ is $\mathfrak{a}$-cofinite if and only if $\mathrm{Supp}_RM\subseteq\mathrm{V}(\mathfrak{a})$ and $\mathrm{Ext}^i_R(R/\mathfrak{a},M)$ are finitely generated for $i\leq n$.}}
\end{lem}
\begin{proof} `Only if' part is obvious.

`If' part. Let $N$ be a finitely generated submodule of $M$ such that $\mathrm{dim}_RM/N\leq n$. Then the exact sequence
\begin{center}
$\mathrm{Hom}_R(R/\mathfrak{a},M)\rightarrow\mathrm{Hom}_R(R/\mathfrak{a},M/N)
\rightarrow\mathrm{Ext}^1_R(R/\mathfrak{a},N)\rightarrow\mathrm{Ext}^1_R(R/\mathfrak{a},M)\rightarrow \mathrm{Ext}^1_R(R/\mathfrak{a},M/N)\rightarrow\mathrm{Ext}^2_R(R/\mathfrak{a},N)$
\end{center}implies that $\mathrm{Ext}^i_R(R/\mathfrak{a},M/N)$ are finitely generated for all $i\leq n$. Hence $M/N$ is $\mathfrak{a}$-cofinite by \cite[Lemma 2.1]{LM1} and \cite[Proposition 2.6]{BNS} and then $M$ is $\mathfrak{a}$-cofinite.
\end{proof}

Using similar arguments, one can prove the dual version of Lemma \ref{lem:0.6}.

\begin{lem}\label{lem:0.7}{\it{Let $\mathcal{S}$ be a Serre subcategory of $R$-modules and $\mathfrak{a},\mathfrak{b}$ two ideals of $R$ with $\mathfrak{b}\subseteq\mathfrak{a}$, and let $X\in\mathrm{D}_\sqsupset(R)$ and $s\geq0,t\geq\mathrm{inf}X$ such that

$(1)$ $\mathrm{Tor}_{s+t}^R(R/\mathfrak{a},X)$ is in $\mathcal{S}$;

$(2)$ $\mathrm{Tor}_{s+1+i}^R(R/\mathfrak{a},\mathrm{H}_{t-i}^\mathfrak{b}(X))$ is in $\mathcal{S}$ for all $1\leq i\leq t-\mathrm{inf}X$;

$(3)$ $\mathrm{Tor}_{s-1-i}^R(R/\mathfrak{a},\mathrm{H}_{t+i}^\mathfrak{b}(X))$ is in $\mathcal{S}$ for all $1\leq i\leq s-1$.\\
Then the $R$-module $\mathrm{Tor}_{s}^R(R/\mathfrak{a},\mathrm{H}_{t}^\mathfrak{b}(X))$ belongs to $\mathcal{S}$.}}
\end{lem}

The \emph{arithmetic
rank} of the ideal $\mathfrak{a}$, denoted by
$\mathrm{ara}(\mathfrak{a})$, is the least number of elements of $R$ required to generate an ideal which has
the same radical as $\mathfrak{a}$, i.e.,
\begin{center}$\mathrm{ara}(\mathfrak{a})=\mathrm{min}\{n\geq0\hspace{0.03cm}|\hspace{0.03cm}\exists\ a_1,\cdots,a_n\in R\ \textrm{with}\ \mathrm{Rad}(a_1,\cdots,a_n)=\mathrm{Rad}(\mathfrak{a})\}$.\end{center}

For each
$R$-module $M$, set
\begin{center}$\mathrm{cd}(\mathfrak{a},M)=\mathrm{sup}\{n\in\mathbb{Z}\hspace{0.03cm}|\hspace{0.03cm}\mathrm{H}_\mathfrak{a}^n(M)\neq0\}$.\end{center}
The \emph{cohomological dimension} of the ideal $\mathfrak{a}$ is
\begin{center}$\mathrm{cd}(\mathfrak{a},R)=\mathrm{sup}\{\mathrm{cd}(\mathfrak{a},M)\hspace{0.03cm}|\hspace{0.03cm}M\ \textrm{is\ an}\ R\textrm{-module}\}$.\end{center}

We now present the first main theorem of this section, which is a nice generalization of \cite[Theorem 1]{EK} and \cite[Corollary 2.4.11]{F}.

\begin{thm}\label{lem:1.1}{\it{Let $\mathfrak{a}$ be a proper ideal of $R$ and $X\in \mathrm{D}_\sqsubset(R)$. If $\mathrm{H}_i(X)$ is $\mathfrak{a}$-cofinite for every $i\in\mathbb{Z}$ then
 $X$ is $\mathfrak{a}$-cofinite. The converse holds when either of the following holds:

 $(1)$ $\mathrm{dim}R/\mathfrak{a}\leq1$.

$(2)$ $\mathrm{ara}(\mathfrak{a})\leq1$.

$(3)$ $\mathrm{dim}R\leq2$.

$(4)$ $\mathrm{dim}_R\mathrm{H}_i(X)\leq1$ for every $i\in\mathbb{Z}$.

 $(5)$ $\mathrm{H}_i(X)$ is in $FD_{\leq 1}$ for every $i\in\mathbb{Z}$.

 $(6)$ $\mathrm{cd}(\mathfrak{a},R)\leq1$ and $X\in\mathrm{D}_\square(R)$.}}
\end{thm}
\begin{proof} The first statement follows from \cite[Lemma 2.4.4]{F}.

If the conditions (1)--(5) hold, then by \cite[Proposition 2.6]{BNS},  \cite[Theorem 2.3]{LM1}, \cite[Lemma 2.2.13]{F}, \cite[Corollary 2.4]{NS} and Lemma \ref{lem:6.6}, one has $\mathrm{Hom}_R(R/\mathfrak{a},\mathrm{H}_i(X))$ and $\mathrm{Ext}^1_R(R/\mathfrak{a},\mathrm{H}_i(X))$ are finitely generated if and only if $\mathrm{H}_i(X)$ are $\mathfrak{a}$-cofinite for all $i$. Thus the second statement follows from Lemma \ref{lem:0.6} by letting $\mathfrak{b}={0}$.
If $\mathrm{cd}(\mathfrak{a},R)\leq1$, then $\mathrm{Tor}_i^R(R/\mathfrak{a},X)$ is finitely generated for all $i$ by \cite[Theorem 7.4]{WW}.
Hence Lemma \ref{lem:0.7} implies that $R/\mathfrak{a}\otimes_R\mathrm{H}_i(X)$ and $\mathrm{Tor}_1^R(R/\mathfrak{a},\mathrm{H}_i(X))$ are finitely generated. Consequently,
$\mathrm{H}_{i}(X)$ is $\mathfrak{a}$-cofinite for $i\in\mathbb{Z}$ by \cite[Corollary 2.2.17]{F} and \cite[Theorem 2.5 and Corollary 3.2]{GM}.
\end{proof}

The following result shows when the $R$-module $\mathrm{Ext}^{s}_R(R/\mathfrak{a},X)$ belongs to $\mathcal{S}$.

\begin{lem}\label{lem:1.3''}{\it{Let $\mathcal{S}$ be a Serre subcategory of $R$-modules and $\mathfrak{a},\mathfrak{b}$ two ideals of $R$ with $\mathfrak{b}\subseteq\mathfrak{a}$, and let $s$ be an integer and $X\in \mathrm{D}_\sqsubset(R)$ such that $\mathrm{Ext}^{s-i}_R(R/\mathfrak{a},\mathrm{H}_\mathfrak{b}^{i}(X))\in\mathcal{S}$ for
all $i\leq s$, then $\mathrm{Ext}^i_R(R/\mathfrak{a},X)\in\mathcal{S}$ for all $i\leq s$.}}
\end{lem}
\begin{proof} We may assume that $i,s\geq-\mathrm{sup}X$. Consider the first spectral sequence in Lemma \ref{lem:0.2}. For
 $-\mathrm{sup}X\leq i\leq s$, there exists a finite filtration
\begin{center}$
 0=U^{-i-\mathrm{sup}X-1}\subseteq U^{-i-\mathrm{sup}X}\subseteq\cdots\subseteq U^0=\mathrm{Ext}^{i}_R(R/\mathfrak{a},X)$,
\end{center}such that $U^{p}/U^{p-1}\cong E^\infty_{p,-i-p}$ for $-i\leq p+\mathrm{sup}X$. Since $E^\infty_{p,-i-p}$ is a subquotient of $E^2_{p,-i-p}$, it follows that $E^\infty_{p,-i-p}\in\mathcal{S}$ for $i+p\leq s$. A successive use of the exact sequence
\begin{center}$
 0\rightarrow U^{p-1}\rightarrow U^{p}\rightarrow U^{p}/U^{p-1}\rightarrow0$,
\end{center} implies that $\mathrm{Ext}^i_R(R/\mathfrak{a},X)\in\mathcal{S}$ for
all $i\leq s$.
\end{proof}

The next lemma
gives some conditions to examine $\mathfrak{a}$-cofiniteness of complexes in $\mathrm{D}_\sqsupset(R)$.

\begin{lem}\label{lem:1.3}{\it{Let $\mathfrak{a}$ be a proper ideal of $R$ with $\mathrm{pd}_RR/\mathfrak{a}\leq1$ and $\mathrm{cd}(\mathfrak{a},R)\leq1$ $($e.g. $R$ is a hereditary ring$)$. For an $R$-complex $X\in\mathrm{D}_\sqsupset(R)$, the following statements are equivalent:

 $(1)$ $X$ is $\mathfrak{a}$-cofinite;

 $(2)$ $\mathrm{Supp}_RX\subseteq\mathrm{V}(\mathfrak{a})$ and $\mathrm{Tor}^R_j(R/\mathfrak{a},X)$ are finitely generated for all $j$;

 $(3)$ $\mathrm{H}_i(X)$ is $\mathfrak{a}$-cofinite for every $i\in\mathbb{Z}$.}}
\end{lem}
\begin{proof} (1) $\Rightarrow$ (2) Let $P\stackrel{\simeq}\rightarrow R/\mathfrak{a}$ be a projective resolution. By \cite[Theorem 10.85]{R}, for any $i$, there exists a short exact sequence\begin{center}
$0\rightarrow\mathrm{Ext}^1_R(\mathrm{H}_0(P),\mathrm{H}_{i-1}(X))
\rightarrow\mathrm{H}^{1}(\mathrm{Hom}_R(P,\Sigma^{-i+1}X))\rightarrow\mathrm{Hom}_R(\mathrm{H}_0(P),\mathrm{H}_{i}(X))\rightarrow0$.\end{center}
Since $\mathrm{H}^{n}(\mathrm{Hom}_R(P,X))$ are finitely generated for all $n$, it follows from the above sequence that $\mathrm{Hom}_R(R/\mathfrak{a},\mathrm{H}_{i}(X))$ and $\mathrm{Ext}^1_R(R/\mathfrak{a},\mathrm{H}_{i}(X))$ are finitely generated. Thus $\mathrm{Tor}_j^R(R/\mathfrak{a},\mathrm{H}_i(X))$ are finitely generated for all $i,j$ by
\cite[Corollary 2.2.17]{F} and \cite[Theorem 2.5 and Corollary 3.2]{GM}. For $j\geq\mathrm{inf}X$, consider the second spectral sequence in Lemma \ref{lem:0.3}. There exists a finite filtration\begin{center}
$0=U^{-1}\subseteq U^{0}\subseteq\cdots \subseteq U^{j-\mathrm{inf}X}=\mathrm{Tor}_{j}^R(R/\mathfrak{a},X)$,
\end{center}such that $U^{p}/U^{p-1}\cong E^\infty_{p,j-p}$ for $p\leq j-\mathrm{inf}X$. Since $E^\infty_{p,j-p}$
is a subquotient of $E^2_{p,j-p}$, $E^\infty_{p,j-p}$ is finitely generated for $0\leq p\leq j$. A successive use of the short exact sequence
\begin{center}$0\rightarrow U^{p-1}\rightarrow U^p\rightarrow U^p
/U^{p-1}\rightarrow0$\end{center}
 implies that
$\mathrm{Tor}_j^R(R/\mathfrak{a},X)$ is finitely generated.

(2) $\Rightarrow$ (3) This follows from Lemma \ref{lem:0.7} as $\mathrm{H}_i(X)$ is $\mathfrak{a}$-cofinite if and only if $R/\mathfrak{a}\otimes_R\mathrm{H}_i(X)$ and $\mathrm{Tor}_1^R(R/\mathfrak{a},\mathrm{H}_i(X))$ are finitely generated for every $i\in\mathbb{Z}$ by \cite[Corollary 2.2.17]{F} and \cite[Theorem 2.5 and Corollary 3.2]{GM}.

(3) $\Rightarrow$ (1) Let $P\stackrel{\simeq}\rightarrow R/\mathfrak{a}$ be a projective resolution. By \cite[Theorem 10.85]{R}, for any $n$, there exists a short exact sequence\begin{center}
$0\rightarrow\mathrm{Ext}^1_R(\mathrm{H}_{-n+1}(\Sigma^{-n+1}P),\mathrm{H}_{n-1}(X))
\rightarrow\mathrm{H}^{1}(\mathrm{Hom}_R(\Sigma^{-n+1}P,X))\rightarrow\mathrm{Hom}_R(\mathrm{H}_{-n+1}(\Sigma^{-n+1}P),\mathrm{H}_{n}(X))
\rightarrow0$.\end{center}
As $\mathrm{Ext}^1_R(R/\mathfrak{a},\mathrm{H}_{i}(X)),\mathrm{Hom}_R(R/\mathfrak{a},\mathrm{H}_{i}(X))$ are finitely generated for all $i$, it follows from the above sequence that $\mathrm{Ext}^n_R(R/\mathfrak{a},X)$ is finitely generated, and hence $X$ is $\mathfrak{a}$-cofinite.
\end{proof}

A local ring is \emph{Cohen-Macaulay} if some system of parameters is a regular sequence. A ring $R$ is \emph{Cohen-Macaulay} if $R_\mathfrak{m}$ is Cohen-Macaulay for every maximal ideal $\mathfrak{m}$ of $R$. A nonzero finitely generated $R$-module
$M$ is \emph{perfect} if $\mathrm{pd}_RM=\mathrm{min}\{i\hspace{0.03cm}|\hspace{0.03cm}\mathrm{Ext}^i_R(M,R)\neq0\}$. An ideal $\mathfrak{a}$ is called \emph{perfect} if $R/\mathfrak{a}$
is a perfect module. The next theorem is the second main result of this section, which
answers Question 3 in introduction completely.

\begin{thm}\label{lem:1.4}{\it{Let $\mathfrak{a}$ be a proper ideal of $R$. Suppose that either of the next conditions holds:

$(1)$ $\mathrm{pd}_RR/\mathfrak{a}\leq1$ and $\mathrm{ara}(\mathfrak{a})\leq1$.

$(2)$ $\mathfrak{a}$ is a principal ideal generated by a nonzerodivisor.

$(3)$ $R$ is a Cohen-Macaulay ring and $\mathfrak{a}$ a perfect ideal with $\mathrm{dim}R/\mathfrak{a}\leq1$ and $\mathrm{pd}_RR/\mathfrak{a}<\infty$.\\
Then an $R$-complex $X\in \mathrm{D}(R)$ is $\mathfrak{a}$-cofinite if and only if every $\mathrm{H}_i(X)$ is $\mathfrak{a}$-cofinite.}}
\end{thm}
\begin{proof} Assume that (1) and (2) hold. For any $n$, set $X_{>n}:\cdots\rightarrow X_{n+1}\rightarrow\mathrm{Im}d_{n+1}\rightarrow0$ and $X_{\leq n}:0\rightarrow X_n/\mathrm{Im}d_{n+1}\rightarrow X_{n-1}\rightarrow\cdots$. We have a short exact sequence \begin{center}$0\rightarrow X_{>n}\rightarrow X\rightarrow X_{\leq n}\rightarrow0$\end{center}with $X_{>n}\in\mathrm{D}_\sqsupset(R)$ and $X_{\leq n}\in\mathrm{D}_\sqsubset(R)$.
If $\mathrm{H}_i(X)$ is $\mathfrak{a}$-cofinite for each $i\in\mathbb{Z}$, then $X_{\leq n}$ is $\mathfrak{a}$-cofinite by Theorem \ref{lem:1.1} and $X_{>n}$ is $\mathfrak{a}$-cofinite by Lemmas \ref{lem:1.3} as $\mathrm{cd}(\mathfrak{a},R)\leq\mathrm{ara}(\mathfrak{a})$. Conversely, let $X$
be $\mathfrak{a}$-cofinite and $P\stackrel{\simeq}\rightarrow R/\mathfrak{a}$ a projective resolution. For any $i$, \cite[Theorem 10.85]{R} yields a short exact sequence\begin{center}
$0\rightarrow\mathrm{Ext}^1_R(\mathrm{H}_0(P),\mathrm{H}_{i-1}(X))
\rightarrow\mathrm{H}^{1}(\mathrm{Hom}_R(P,\Sigma^{-i+1}X))\rightarrow\mathrm{Hom}_R(\mathrm{H}_0(P),\mathrm{H}_{i}(X))\rightarrow0$.\end{center}
Since $\mathrm{H}^{n}(\mathrm{Hom}_R(P,X))$ are finitely generated for all $n$, it follows from the above sequence that $\mathrm{Hom}_R(R/\mathfrak{a},\mathrm{H}_{i}(X))$ and $\mathrm{Ext}^1_R(R/\mathfrak{a},\mathrm{H}_{i}(X))$ are finitely generated. Hence \cite[Corollary 2.2.13]{F} implies that $\mathrm{H}_i(X)$ is $\mathfrak{a}$-cofinite for every $i\in\mathbb{Z}$.

Assume that (3) holds. Let $\mathrm{H}_i(X)$ is $\mathfrak{a}$-cofinite for every $i\in\mathbb{Z}$. Apply the functor $\mathrm{RHom}_R(R/\mathfrak{a},-)$ to the above exact sequence to get an exact triangle
\begin{center}$\mathrm{RHom}_R(R/\mathfrak{a},X_{> n})\rightarrow\mathrm{RHom}_R(R/\mathfrak{a},X)\rightarrow
\mathrm{RHom}_R(R/\mathfrak{a},X_{\leq n})\rightsquigarrow$.\end{center}For any $i\in\mathbb{Z}$, choose $n=i+\mathrm{pd}_RR/\mathfrak{a}$. From the long exact cohomology sequence, we have
\begin{center}$0=\mathrm{H}_{i}(\mathrm{RHom}_R(R/\mathfrak{a},X_{> n}))\rightarrow\mathrm{H}_{i}(\mathrm{RHom}_R(R/\mathfrak{a},X))\rightarrow
\mathrm{H}_{i}(\mathrm{RHom}_R(R/\mathfrak{a},X_{\leq n}))\rightarrow\mathrm{H}_{i-1}(\mathrm{RHom}_R(R/\mathfrak{a},X_{> n}))=0$.\end{center}By Theorem \ref{lem:1.1}, one has $X_{\leq n}$ is $\mathfrak{a}$-cofinite, so $\mathrm{H}_{i}(\mathrm{RHom}_R(R/\mathfrak{a},X))\cong\mathrm{H}_{i}(\mathrm{RHom}_R(R/\mathfrak{a},X_{\leq n}))$ are finitely generated for all $i$. Thus $X$ is $\mathfrak{a}$-cofinite. Conversely,
let $X$ be $\mathfrak{a}$-cofinite and $X\stackrel{\simeq}\rightarrow I$ a minimal semi-injective resolution. If $\mathrm{dim}R/\mathfrak{a}=0$, then there is a finite filtration
$0=N_{0}\subseteq N_{1}\subseteq \cdots\subseteq N_{r}=R/\mathfrak{a}$
such that $N_{j}/N_{j-1}\cong R/\mathfrak{m}$ for $1\leq j\leq r$. The exact sequence $0\rightarrow R/\mathfrak{m}\rightarrow N_2\rightarrow R/\mathfrak{m}\rightarrow 0$ yields a commutative diagram
\begin{center}$\xymatrix@C=20pt@R=20pt{
  & \vdots\ar[d] & \vdots\ar[d] & \vdots\ar[d] & \\
0\ar[r]&\mathrm{Hom}_{R}(R/\mathfrak{m},I_{i})\ar[r]\ar[d]&\mathrm{Hom}_{R}(N_{2},I_{i})\ar[d]\ar[r]&\mathrm{Hom}_{R}(R/\mathfrak{m},I_{i})\ar[d]\ar[r]&0\\
0\ar[r]&\mathrm{Hom}_{R}(R/\mathfrak{m},I_{i-1})\ar[r]\ar[d]&\mathrm{Hom}_{R}(N_{2},I_{i-1})\ar[r]\ar[d]&\mathrm{Hom}_{R}(R/\mathfrak{m},I_{i-1})\ar[r]\ar[d]&0.\\
  & \vdots & \vdots & \vdots & }$
\end{center}As $\mathrm{Hom}_{R}(R/\mathfrak{m},I)$ has zero differentials, it follows that
$\mathrm{Hom}_R(N_{2},I)$ has zero differentials. By repeating this process, we see that $\mathrm{Hom}_R(R/\mathfrak{a},I)$ has zero differentials. Therefore,
$\mathrm{H}_i(\mathrm{RHom}_R(R/\mathfrak{a},X))\cong\mathrm{H}_i(\mathrm{Hom}_R(R/\mathfrak{a},I))
\cong\mathrm{Hom}_R(R/\mathfrak{a},I_i)$ is finitely generated. Hence
$I_i$ is artinian and $\mathfrak{a}$-cofinite by \cite[Proposition 4.1]{LM}, and then $\mathrm{H}_i(X)$ is $\mathfrak{a}$-cofinite for every $i\in\mathbb{Z}$.
Now let $\mathrm{dim}R/\mathfrak{a}=1$. Then \cite[Theorems 2.1.2(a) and 2.1.5(a)]{BH} imply that $\mathrm{Ass}_{R}R/\mathfrak{a}=\{\mathfrak{p}_1,\cdots,\mathfrak{p}_s\}$, where $\mathfrak{p}_j\in\mathrm{V}(\mathfrak{a})$ with $\mathrm{dim}R/\mathfrak{p}_j=1$ for $j=1,\cdots,s$.
By Prime Avoidance Theorem, there exists $x\in \mathfrak{m}\backslash(\mathfrak{p}_{1}\cup\cdots\cup\mathfrak{p}_{s})$ that is
 $R/\mathfrak{a}$-regular and $\mathrm{dim}R/(\mathfrak{a}+xR)=0$. Consider the short exact sequence
\begin{center}$\xymatrix@C=20pt@R=20pt{
0\ar[r] & R/\mathfrak{a}\ar[r]^{x} & R/\mathfrak{a}\ar[r] & R/(\mathfrak{a}+xR)\ar[r] &0, }$
\end{center}
which induces the following exact sequence
\begin{center}
$\cdots\rightarrow\mathrm{Ext}^{i-1}_R(R/\mathfrak{a},X)\rightarrow\mathrm{Ext}^i_R(R/(\mathfrak{a}+xR),X)
\rightarrow\mathrm{Ext}^i_R(R/\mathfrak{a},X)\stackrel{x}\rightarrow \mathrm{Ext}^{i}_R(R/\mathfrak{a},X)\rightarrow\cdots$.
\end{center}
Then $\mathrm{RHom}_{R}(R/(\mathfrak{a}+xR),\mathrm{RHom}_{R}(R/xR,X))\simeq
\mathrm{RHom}_{R}(R/(\mathfrak{a}+xR),X)\in\mathrm{D}^\mathrm{f}(R)$.
Consider the projective resolution $P:0\rightarrow R\stackrel{x}\rightarrow R\rightarrow0$ of $R/xR$. For any $i$,
\cite[Theorem 10.85]{R} yields a short exact sequence\begin{center}
$0\rightarrow\mathrm{Ext}^1_R(R/xR,\mathrm{H}_{i-1}(X))
\rightarrow\mathrm{H}^{1}(\mathrm{Hom}_R(P,\Sigma^{-i+1}X))\rightarrow\mathrm{Hom}_R(R/xR,\mathrm{H}_{i}(X))\rightarrow0$.\end{center}
Since each $\mathrm{H}_{n}(\mathrm{Hom}_R(P,X))$ is artinian $\mathfrak{a}+xR$-cofinite by the preceding proof, it follows that
$\mathrm{Hom}_R(R/xR,\mathrm{H}_{i}(X))$ and $\mathrm{Ext}^1_R(R/xR,\mathrm{H}_{i}(X))$ is artinian $\mathfrak{a}+xR$-cofinite, so $\mathrm{Ext}^j_{R/xR}(R/\mathfrak{a}+xR,\mathrm{Hom}_R(R/xR,\mathrm{H}_{i}(X)))\cong\mathrm{Ext}^j_R(R/\mathfrak{a}+xR,\mathrm{H}_{i}(X))$ are finitely generated for all $j\geq0$ and $i\in\mathbb{Z}$. As $xR\cong R$ is projective and $\mathrm{Supp}_RxR/xR\cap\mathfrak{a}\subseteq\{\mathfrak{m}\}$, $\mathrm{Ext}^j_R(R/\mathfrak{a},\mathrm{Hom}_R(xR,\mathrm{H}_{i}(X)))\cong
\mathrm{Ext}^j_R(xR/xR\cap\mathfrak{a},\mathrm{H}_{i}(X))$ are finitely generated for all $j\geq0$ and $i\in\mathbb{Z}$.
Hence the following exact sequence \begin{center}
$0\rightarrow\mathrm{Hom}_R(R/xR,\mathrm{H}_{i}(X))
\rightarrow\mathrm{H}_{i}(X)\rightarrow\mathrm{Hom}_R(xR,\mathrm{H}_{i}(X))\rightarrow\mathrm{Ext}^1_R(R/xR,\mathrm{H}_{i}(X))\rightarrow0$\end{center} implies that $\mathrm{Ext}^j_R(R/\mathfrak{a},\mathrm{H}_{i}(X))$ is finitely generated for $j\geq0$. Therefore, $\mathrm{H}_i(X)$ is $\mathfrak{a}$-cofinite for every $i\in\mathbb{Z}$.
\end{proof}

\begin{cor}\label{lem:1.4'}{\it{Let $(R,\mathfrak{m})$ be a regular local ring and $\mathfrak{a}$ a proper ideal of $R$ such that either $\mathrm{dim}R\leq2$ and $\mathfrak{a}$ is perfect, or $\mathfrak{a}$ is perfect with $\mathrm{dim}R/\mathfrak{a}\leq1$, or $\mathrm{ara}(\mathfrak{a})\leq1$, then an $R$-complex $X$ is $\mathfrak{a}$-cofinite if and only if $\mathrm{H}_i(X)$ is $\mathfrak{a}$-cofinite for every $i\in\mathbb{Z}$.}}
\end{cor}

\bigskip
\section{\bf Cofiniteness of complexes in the case $d\geq3$}
The task of this section is to give an answer of Question 3 in the cases $\mathrm{dim}R/\mathfrak{a}\geq2$, or $\mathrm{dim}R\geq3$, or $\mathrm{ara}(\mathfrak{a})\geq2$. To do this, we first generalize the Bahmanpour, Naghipour and Sedghi's result \cite[Theorem 3.4]{BNS1} to not necessarily local rings.

\begin{lem}\label{lem:2.1}{\it{Let $M$ be an $R$-module with $\mathrm{dim}_RM\leq2$. Then
$M$ is $\mathfrak{a}$-cofinite if and only if $\mathrm{Supp}_RM\subseteq\mathrm{V}(\mathfrak{a})$ and  $\mathrm{Ext}^i_R(R/\mathfrak{a},M)$ are finitely generated for $i\leq2$.}}
\end{lem}
\begin{proof} `Only if' part is obvious.

`If' part. By \cite[Proposition 2.6]{BNS} we may assume $\mathrm{dim}_RM=2$ and let
$t=\mathrm{ara}_M(\mathfrak{a})$. If $t=0$, then $\mathfrak{a}^n\subseteq\mathrm{Ann}_RM$ for
some $n\geq0$, and so $M=(0:_M\mathfrak{a}^n)$ is finitely generated by \cite[Lemma 2.1]{ANS} and the assertion follows. Next, assume that $t>0$, and let\begin{center}
$T_M=\{\mathfrak{p}\in\mathrm{Supp}_RM\hspace{0.03cm}|\hspace{0.03cm}\mathrm{dim}R/\mathfrak{p}=2\}$.
\end{center}
As $\mathrm{Ass}_R\mathrm{Hom}_R(R/\mathfrak{a},M)=\mathrm{Ass}_RM$ is finite, the set $T_M$ is finite. Also for each $\mathfrak{p}\in T_M$, the $R_\mathfrak{p}$-module $\mathrm{Hom}_{R_\mathfrak{p}}(R_\mathfrak{p}/\mathfrak{a}R_\mathfrak{p},M_\mathfrak{p})$ is finitely generated and $M_\mathfrak{p}$ is an $\mathfrak{a}R_\mathfrak{p}$-torsion $R_\mathfrak{p}$-module with $\mathrm{Supp}_{R_\mathfrak{p}}M_\mathfrak{p}\subseteq\mathrm{V}(\mathfrak{p}R_\mathfrak{p})$, it follows from \cite[Proposition 4.1]{LM}
that $M_\mathfrak{p}$ is an artinian $\mathfrak{a}R_\mathfrak{p}$-cofinite $R_\mathfrak{p}$-module. Let
$T_M=\{\mathfrak{p}_1,\cdots\mathfrak{,p}_n\}$.
It follows from \cite[Lemma 2.5]{BN} that $\mathrm{V}(\mathfrak{a}R_{\mathfrak{p}_j})\cap\mathrm{Att}_{R_{\mathfrak{p}_j}}M_{\mathfrak{p}_j}
\subseteq\mathrm{V}(\mathfrak{p}_jR_{\mathfrak{p}_j})$ for $j=1,\cdots,n$. Set \begin{center}
$\mathrm{U}_M=\bigcup_{j=1}^n\{\mathfrak{q}\in\mathrm{Spec}R\hspace{0.03cm}|\hspace{0.03cm}\mathfrak{q}R_{\mathfrak{p}_j}
\in\mathrm{Att}_{R_{\mathfrak{p}_j}}M_{\mathfrak{p}_j}\}$.
\end{center}Then $\mathrm{U}_M\cap\mathrm{V}(\mathfrak{a})\subseteq T_M$. Since $t=\mathrm{ara}_M(\mathfrak{a})\geq 1$, there exist $y_1,\cdots,y_t\in\mathfrak{a}$
such that\begin{center}
$\mathrm{Rad}(\mathfrak{a}+\mathrm{Ann}_RM/\mathrm{Ann}_RM)=\mathrm{Rad}((y_1,\cdots,y_t)+\mathrm{Ann}_RM/\mathrm{Ann}_RM)$.
\end{center}Since $\mathfrak{a}\nsubseteq\bigcup_{\mathfrak{q}\in\mathrm{U}_M\backslash\mathrm{V}(\mathfrak{a})}\mathfrak{q}$, it follows that $(y_1,\cdots,y_t)+\mathrm{Ann}_RM\nsubseteq\bigcup_{\mathfrak{q}\in\mathrm{U}_M\backslash\mathrm{V}(\mathfrak{a})}\mathfrak{q}$.
On the other hand, for each $\mathfrak{q}\in\mathrm{U}_M$ we have $\mathfrak{q}R_{\mathfrak{p}_j}
\in\mathrm{Att}_{R_{\mathfrak{p}_j}}M_{\mathfrak{p}_j}$ for some $1\leq j\leq n$. Thus
\begin{center}
$(\mathrm{Ann}_RM)R_{\mathfrak{p}_j}\subseteq\mathrm{Ann}_{R_{\mathfrak{p}_j}}M_{\mathfrak{p}_j}\subseteq \mathfrak{q}R_{\mathfrak{p}_j}$,
\end{center}and so $\mathrm{Ann}_RM\subseteq\mathfrak{q}$. Consequently, $(y_1,\cdots,y_t)\nsubseteq\bigcup_{\mathfrak{q}\in\mathrm{U}_M\backslash\mathrm{V}(\mathfrak{a})}\mathfrak{q}$ as $\mathrm{Ann}_RM\subseteq\bigcap_{\mathfrak{q}\in\mathrm{U}_M\backslash\mathrm{V}(\mathfrak{a})}\mathfrak{q}$. Hence \cite[Ex.16.8]{M}
provides an element $a_1\in(y_2,\cdots,y_t)$ such that $y_1+a_1\not\in\bigcup_{\mathfrak{q}\in\mathrm{U}_M\backslash\mathrm{V}(\mathfrak{a})}\mathfrak{q}$. Set $x_1=y_1+a_1$. Then $x_1\in\mathfrak{a}$ and
\begin{center}
$\mathrm{Rad}(\mathfrak{a}+\mathrm{Ann}_RM/\mathrm{Ann}_RM)=\mathrm{Rad}((x_1,y_2,\cdots,y_t)+\mathrm{Ann}_RM/\mathrm{Ann}_RM)$.
\end{center}Let $M_1=(0:_Mx_1)$. Then $\mathrm{ara}_{M_1}(\mathfrak{a})=
\mathrm{ara}(\mathfrak{a}+\mathrm{Ann}_{R}M_{1}/\mathrm{Ann}_{R}M_{1})\leq t-1$ as $x_1\in\mathrm{Ann}_RM_1$. The reasoning in the preceding  applied to $M_1$, there exists $x_2\in\mathfrak{a}$ such that $\mathrm{ara}_{M_2}(\mathfrak{a})=
\mathrm{Rad}((y_3,\cdots,y_t)+\mathrm{Ann}_RM_2/\mathrm{Ann}_RM_2)\leq t-2$.
 Continuing this process, one obtains elements $x_1,\cdots,x_t\in\mathfrak{a}$ and the sequences
\begin{center}
$0\rightarrow M_{i}\rightarrow M_{i-1}\rightarrow x_iM_{i-1}\rightarrow0$,
\end{center}where $M_{0}=M$ and $M_{i}=(0:_{M_{i-1}}x_i)$ such that $\mathrm{ara}_{M_i}(\mathfrak{a})\leq t-i$ for $i=1,\cdots,t$, and these
exact sequences induce an exact sequence
\begin{center}
$0\rightarrow M_t\rightarrow M\rightarrow (x_1,\cdots,x_t)M\rightarrow0$.
\end{center}As $\mathrm{Hom}_R(R/\mathfrak{a},M_t)$ is finitely generated and $\mathrm{ara}_{M_t}(\mathfrak{a})=0$, one has $M_t$ is $\mathfrak{a}$-cofinite. Moreover, the above sequence implies that
$\mathrm{Ext}^1_R(R/\mathfrak{a},(x_1,\cdots,x_t)M)$ and $\mathrm{Ext}^2_R(R/\mathfrak{a},(x_1,\cdots,x_t)M)$ are finitely generated. Also the exact sequence \begin{center}$0\rightarrow (x_1,\cdots,x_t)M\rightarrow M\rightarrow M/(x_1,\cdots,x_t)M\rightarrow0$ \end{center} yields that $\mathrm{Hom}_R(R/\mathfrak{a},M/(x_1,\cdots,x_t)M)$ and $\mathrm{Ext}^1_R(R/\mathfrak{a},M/(x_1,\cdots,x_t)M)$ are finitely generated. On the other hand, for $i=1,\cdots,t$, the exact sequence \begin{center}
$0\rightarrow x_iM_{i-1}\rightarrow (x_1,\cdots,x_i)M\rightarrow (x_1,\cdots,x_{i-1})M\rightarrow0$
\end{center}induces the following commutative diagram
\begin{center}$\xymatrix@C=20pt@R=20pt{
     &  0\ar[d]  & 0 \ar[d] &  \\
       0\ar[r]&x_iM_{i-1}\ar[r] \ar[d]& M_{i-1}\ar[d] \ar[r]& M_{i-1}/x_iM_{i-1}\ar[d]^\cong\ar[r]&0\\
      0 \ar[r] &(x_1,\cdots,x_i)M \ar[d] \ar[r] &M  \ar[d]\ar[r] &M/(x_1,\cdots,x_i)M  \ar[r] & 0  \\
     &    (x_1,\cdots,x_{i-1})M  \ar[d]\ar@{=}[r] & (x_1,\cdots,x_{i-1})M\ar[d]  \\
    &  0& \hspace{0.15cm}0.  &}$
\end{center}Let $T_{M_{t-1}}=\{\mathfrak{p}\in\mathrm{Supp}_RM_{t-1}\hspace{0.03cm}|\hspace{0.03cm}\mathrm{dim}R/\mathfrak{p}=2\}$. If $T_{M_{t-1}}=\emptyset$, then $\mathrm{dim}_{R}M_{t-1}\leq 1$, it follows that $\mathrm{dim}_{R}M/(x_1,\cdots,x_t)M\leq 1$. Hence $M/(x_1,\cdots,x_t)M$ is $\mathfrak{a}$-cofinite by \cite[Proposition 2.6]{BNS}. Now assume that $T_{M_{t-1}}\neq \emptyset$ and set $T_{M_{t-1}}=\{\mathfrak{q}_{1},\cdots,\mathfrak{q}_{m}\}$. Since \begin{center}$\mathrm{Hom}_R(R/\mathfrak{a},M_{t-1}/x_tM_{t-1})\cong\mathrm{Hom}_R(R/\mathfrak{a},M/(x_1,\cdots,x_t)M)$\end{center} is finitely generated, it follows from \cite[Lemma 2.4]{BN} that $(M_{t-1}/x_tM_{t-1})_{\mathfrak{q}_{j}}$ has finite length for $j=1,\cdots,m$. Thus there is a finitely generated submodule $L_{j}$ of $M/(x_1,\cdots,x_t)M$ such that $(M/(x_1,\cdots,x_t)M)_{\mathfrak{q}_{j}}=(L_{j})_{\mathfrak{q}_{j}}$. Set $L=L_{1}+\cdots+L_{m}$. Then $L$ is a finitely generated submodule of $M/(x_1,\cdots,x_t)M$ so that $\mathrm{Supp}_R(M/(x_1,\cdots,x_t)M)/L\subseteq\mathrm{Supp}_RM_{t-1}\backslash\{\mathfrak{q}_{1},\cdots,\mathfrak{q}_{m}\}$ and hence $\mathrm{dim}_R(M/(x_1,\cdots,x_t)M)/L\leq1$. Now the exact sequence \begin{center}$0\rightarrow L\rightarrow M/(x_1,\cdots,x_t)M\rightarrow (M/(x_1,\cdots,x_t)M)/L\rightarrow0$\end{center}induces the following exact sequence
\begin{center}
$\mathrm{Hom}_{R}(R/\mathfrak{a},M/(x_1,\cdots,x_t)M)\rightarrow
\mathrm{Hom}_{R}(R/\mathfrak{a},(M/(x_1,\cdots,x_t)M)/L)\rightarrow\mathrm{Ext}^1_{R}(R/\mathfrak{a},L)\rightarrow
\mathrm{Ext}^1_{R}(R/\mathfrak{a},M/(x_1,\cdots,x_t)M)\rightarrow\mathrm{Ext}^1_{R}(R/\mathfrak{a},(M/(x_1,\cdots,x_t)M)/L)
\rightarrow\mathrm{Ext}^2_{R}(R/\mathfrak{a},L)$.\end{center} Hence $\mathrm{Hom}_{R}(R/\mathfrak{a},(M/(x_1,\cdots,x_t)M)/L)$ and $\mathrm{Ext}^1_{R}(R/\mathfrak{a},(M/(x_1,\cdots,x_t)M)/L)$ are finitely generated, and so $(M/(x_1,\cdots,x_t)M)/L$ is $\mathfrak{a}$-cofinite by \cite[Proposition 2.6]{BNS}. Consequently, $M/(x_1,\cdots,x_t)M$ is $\mathfrak{a}$-cofinite. As $M_t\cong\mathrm{Hom}_{R}(R/(x_1,\cdots,x_t)R,M)$ and $M/(x_1,\cdots,x_t)M$ are $\mathfrak{a}$-cofinite, it follows from \cite[Corollary 3.4]{LM} that $M$ is $\mathfrak{a}$-cofinite, as desired.
\end{proof}

\begin{cor}\label{lem:2.0}{\it{Let $\mathfrak{a}$ be an ideal of $R$ so that either $\mathrm{dim}R/\mathfrak{a}=2$, or $\mathrm{dim}R=3$, or $\mathrm{ara}(\mathfrak{a})=2$. Then $R$-module $M$ is $\mathfrak{a}$-cofinite if and only if $\mathrm{Supp}_RM\subseteq\mathrm{V}(\mathfrak{a})$ and $\mathrm{Ext}^{i}_R(R/\mathfrak{a},M)$ is finitely generated for $i=0,1,2$.}}
\end{cor}
\begin{proof} This follows from Lemma \ref{lem:2.1}, \cite[Theorem 2.3]{NS} and \cite[Corollary 2.2.13]{F}.
\end{proof}

The next results answer Question 3 in the cases $\mathrm{dim}R/\mathfrak{a}=2$, $\mathrm{dim}R=3$ and $\mathrm{ara}(\mathfrak{a})=2$.

\begin{thm}\label{lem:3.1}{\it{Let $\mathfrak{a}$ be an ideal of $R$ and $s$ an integer.
 If either $\mathrm{dim}R/\mathfrak{a}=2$, or $\mathrm{dim}R=3$, or $\mathrm{ara}(\mathfrak{a})=2$, then for an $\mathfrak{a}$-cofinite complex $X\in \mathrm{D}_\sqsubset(R)$, $\mathrm{Hom}_R(R/\mathfrak{a},\mathrm{H}_{i}(X))$ is finitely generated for $i\geq -s$ if and only if $\mathrm{H}_{i}(X)$ is $\mathfrak{a}$-cofinite for $i\geq -s+1$. }}
\end{thm}
\begin{proof} `Only if' part follows from Lemma \ref{lem:0.6}.

`If' part.
We may assume that $s\geq-\mathrm{sup}X$, and use induction on $s$. When $s=-\mathrm{sup}X$, $-\mathrm{sup}X+1$, there
is nothing to prove. Suppose that $s>-\mathrm{sup}X+1$ and that the result has been proved for smaller
values of $s$. By the induction, it only
remains for us to prove that $\mathrm{H}_{-s+1}(X)$ is $\mathfrak{a}$-cofinite. As $\mathrm{Hom}_R(R/\mathfrak{a},\mathrm{H}_{-s}(X))$ is finitely generated, it follows from Lemma \ref{lem:0.6} that $\mathrm{Ext}^1_R(R/\mathfrak{a},\mathrm{H}_{-s+1}(X))$ and $\mathrm{Ext}^2_R(R/\mathfrak{a},\mathrm{H}_{-s+1}(X))$ are finitely generated. Hence Corollary \ref{lem:2.0} implies that $\mathrm{H}_{-s+1}(X)$ is $\mathfrak{a}$-cofinite.
\end{proof}

\begin{cor}\label{lem:3.41}{\it{Let $\mathfrak{a}$ be a proper ideal of a ring $R$ with $\mathrm{ara}(\mathfrak{a})=2$ and $X\in \mathrm{D}_\square(R)$. The following are equivalent:

 $(1)$ $X$ is $\mathfrak{a}$-cofinite and $\mathrm{Hom}_R(R/\mathfrak{a},\mathrm{H}_{i}(X))$ are finitely generated for all $i\in\mathbb{Z}$;

 $(2)$ $X$ is $\mathfrak{a}$-cofinite and $R/\mathfrak{a}\otimes_R\mathrm{H}_{i}(X)$ are finitely generated for all $i\in\mathbb{Z}$;

$(3)$ $\mathrm{H}_i(X)$ are $\mathfrak{a}$-cofinite for all $i\in\mathbb{Z}$.}}
\end{cor}
\begin{proof} (1) $\Leftrightarrow$ (3) follows from Theorem \ref{lem:3.1}.

(3) $\Rightarrow$ (2) follows from \cite[Lemma 2.4.4 and Corollary 2.2.17]{F}.

(2) $\Rightarrow$ (3) We may assume that $i\geq\mathrm{inf}X$. When $i=\mathrm{inf}X$, $\mathrm{inf}X+1$, there
is nothing to prove. Suppose that $i>\mathrm{inf}X$ and the result has been proved for smaller
values of $i+1$. As $R/\mathfrak{a}\otimes_R\mathrm{H}_{i+1}(X)$ is finitely generated, one has $\mathrm{Tor}_1^R(R/\mathfrak{a},\mathrm{H}_{i}(X))$ and $\mathrm{Tor}_2^R(R/\mathfrak{a},\mathrm{H}_{i}(X))$ are finitely generated by Lemma \ref{lem:0.7}. Hence $\mathrm{H}_{i}(X)$ is $\mathfrak{a}$-cofinite by \cite[Corollary 2.2.13]{F}.
\end{proof}

Next, we give an answer of Question 3 in the cases $\mathrm{dim}R/\mathfrak{a}=d\geq3$, or $\mathrm{dim}R/\mathfrak{a}=d-1$, or $\mathrm{ara}(\mathfrak{a})=d-1$. We start with the following useful lemmas.

\begin{lem}\label{lem:3.42'}{\it{Let $\mathfrak{a}$ be a proper ideal of a local ring $R$ with $\mathrm{dim}R=d\geq1$. Then an $\mathfrak{a}$-torsion $R$-module $M$ is $\mathfrak{a}$-cofinite if and only if $\mathrm{Ext}^i_R(R/\mathfrak{a},M)$ is finitely generated for $i\leq d-1$.}}
\end{lem}
\begin{proof} `Only if' part is trivial.

`If' part.
The case $d\leq3$ is by Corollary \ref{lem:2.0} and \cite[Theorem 2.3]{NS}. Now assume that $d\geq4$ and that the result has been proved for smaller
values of $d-1$. Then $\mathrm{ara}(\mathfrak{a})\leq d$. If $\mathrm{ara}(\mathfrak{a})\leq d-1$ then assertion follows from \cite[Corollary 2.2.13]{F}. So let $\mathrm{Rad}(\mathfrak{a})=\mathrm{Rad}(\mathfrak{a}')$ with $\mathfrak{a}'=(a_1,\cdots,a_d)\subseteq R$, and
set $\mathfrak{b}=(a_1,\cdots,a_{d-1})$ and $\mathfrak{c}=(a_d)$. Since
$\mathrm{Ext}^j_{R/\mathfrak{b}}(R/\mathfrak{a}',\mathrm{Hom}_R(R/\mathfrak{b},E))=0$ for $j\geq 1$ whenever $E$ is an injective $R$-module, it follows from \cite[Theorem 10.64]{R} that there is a third quadrant
spectral sequence \begin{center}$\xymatrix@C=10pt@R=5pt{
 \mathrm{Ext}^{p}_{R/\mathfrak{b}}(R/\mathfrak{a}',\mathrm{Ext}^{q}_R(R/\mathfrak{b},M))\ar@{=>}[r]_{\ \ \ \ \ \ \ p}&
 \mathrm{Ext}^{p+q}_R(R/\mathfrak{a}',M).}$\end{center} As $\mathrm{Ext}^i_R(R/\mathfrak{a}',M)$ are finitely generated for all $i\leq d-1$ by \cite[Corollary 2.2.13]{F}, it follows from the above spectral sequence that \begin{center}$\mathrm{Hom}_{R/\mathfrak{b}}(R/\mathfrak{a}',\mathrm{Ext}^{i}_R(R/\mathfrak{b},M))\cong
 \mathrm{Hom}_{R/\mathfrak{b}}(R/\mathfrak{b}\otimes_RR/\mathfrak{c},\mathrm{Ext}^{i}_R(R/\mathfrak{b},M))\cong
 \mathrm{Hom}_{R}(R/\mathfrak{c},\mathrm{Ext}^{i}_R(R/\mathfrak{b},M))$,\end{center}
 \begin{center}$\mathrm{Ext}^1_{R/\mathfrak{b}}(R/\mathfrak{a}',\mathrm{Ext}^{i}_R(R/\mathfrak{b},M))\cong
 \mathrm{Ext}^1_{R/\mathfrak{b}}(R/\mathfrak{b}\otimes_RR/\mathfrak{c},\mathrm{Ext}^{i}_R(R/\mathfrak{b},M))\cong
 \mathrm{Ext}^1_{R}(R/\mathfrak{c},\mathrm{Ext}^{i}_R(R/\mathfrak{b},M))$\end{center}are finitely generated for $i\leq d-2$. Consequently, $\mathrm{Ext}^{i}_R(R/\mathfrak{b},M))$ are $\mathfrak{c}$-cofinite for $i\leq d-2$, that is to say, $\mathrm{Ext}^j_{R/\mathfrak{b}}(R/\mathfrak{a}',\mathrm{Ext}^{i}_R(R/\mathfrak{b},M))\cong\mathrm{Ext}^j_{R}(R/\mathfrak{c},\mathrm{Ext}^{i}_R(R/\mathfrak{b},M))$ are finitely generated for all $j$ and $i\leq d-2$ by \cite[Corollary 10.65]{R}. Thus $\mathrm{Ext}^{i}_R(R/\mathfrak{a},M)$ are finitely generated for all $i$ by the spectral sequence and \cite[Corollary 2.2.13]{F}, as desired.
\end{proof}

The next lemma is a nice generalization of \cite[Theorem 2.3]{LM1} and \cite[Theorem 3.5]{BNS1}.

\begin{lem}\label{lem:5.1}{\it{Let $\mathfrak{a}$ be a proper ideal of $R$ with $\mathrm{dim}R/\mathfrak{a}=d\geq0$. Then an $\mathfrak{a}$-torsion $R$-module $M$ is $\mathfrak{a}$-cofinite if and only if $\mathrm{Ext}^i_R(R/\mathfrak{a},M)$ is finitely generated for $i\leq d$.}}
\end{lem}
\begin{proof} `Only if' part is trivial.

`If' part. If $d\leq2$, then the assertion holds by Corollary \ref{lem:2.0}. Now let $d\geq3$, and suppose that the result holds for $d-1$. If $\mathfrak{a}$ is nilpotent, say $\mathfrak{a}^n=0$ for some integer $n$, then $M=(0:_M\mathfrak{a}^n)$ is finitely generated and so
$M$ is $\mathfrak{a}$-cofinite. Now suppose that $\mathfrak{a}$ is not nilpotent. We can choose
$n>0$ such that $(0 :_R\mathfrak{a}^n)=\Gamma_\mathfrak{a}(R)$. Set $\bar{R}=R/\Gamma_\mathfrak{a}(R)$, $\bar{M}=M/(0:_M\mathfrak{a}^n)$ and $\bar{\mathfrak{a}}$ be the image of $\mathfrak{a}$ in $\bar{R}$. Since $\Gamma_{\bar{\mathfrak{a}}}(\bar{R})=0$, $\bar{\mathfrak{a}}$ contains a nonzerodivisor of $\bar{R}$, and so $\mathrm{dim}R/(\mathfrak{a}+\Gamma_\mathfrak{a}(R))=\mathrm{dim}\bar{R}/\bar{\mathfrak{a}}\leq d-1$. The assumption on
$M$ and \cite[Proposition 1]{DM}
imply that $\mathrm{Ext}^i_R(R/(\mathfrak{a}+\Gamma_\mathfrak{a}(R)),M)$ is finitely generated for $i\leq d$. Note that $(0:_M\mathfrak{a}^n)$ is finitely generated, it follows from the exact sequence $0\rightarrow (0:_M\mathfrak{a}^n)\rightarrow M\rightarrow \bar{M}\rightarrow0$ that $\mathrm{Ext}^i_R(R/(\mathfrak{a}+\Gamma_\mathfrak{a}(R)),\bar{M})$ is finitely generated for $i\leq d$. By induction, $\bar{M}$ is $\mathfrak{a}+\Gamma_\mathfrak{a}(R)$-cofinite since
$\mathrm{Supp}_R\bar{M}\subseteq \mathrm{V}(\mathfrak{a}+\Gamma_\mathfrak{a}(R))$, and hence $\bar{M}$ is $\mathfrak{a}$-cofinite by \cite[Proposition 2]{DM}. Consequently, $M$ is $\mathfrak{a}$-cofinite.
\end{proof}

The next main theorem is a more general version of \cite[Theorem 2.8]{NS}.

\begin{thm}\label{lem:3.43}{\it{Let $d$ be a positive integer and $\mathfrak{a}$ a proper ideal of a local ring $R$ so that either $\mathrm{dim}R=d$, or $\mathrm{dim}R/\mathfrak{a}=d-1$, or $\mathrm{ara}(\mathfrak{a})=d-1$. Then for a complex $X$ of
 $\mathfrak{a}$-cofinite $R$-modules, $\mathrm{H}_i(X)$ are $\mathfrak{a}$-cofinite for all $i\in\mathbb{Z}$ if and only if $\mathrm{Ext}^j_R(R/\mathfrak{a},\mathrm{coker}d_i)$ are finitely generated for all $j\leq d-2$ and all $i\in\mathbb{Z}$.}}
\end{thm}
\begin{proof} Set $B_i=\mathrm{im}d_{i+1}$, $Z_i=\mathrm{ker}d_{i}$, $C_i=\mathrm{coker}d_{i+1}$ and $H_i=\mathrm{H}_{i}(X)$. We have the following commutative diagram:\begin{center} $\xymatrix@C=20pt@R=20pt{
      & & 0\ar[d]  & 0 \ar[d] &  \\
       0\ar[r]& B_i\ar@{=}[d] \ar[r]& Z_i\ar[d]\ar[r] &H_i\ar[d]\ar[r]&0 \\
0\ar[r]& B_i \ar[r]& X_i\ar[d]\ar[r] &C_i\ar[d]\ar[r]&0 \\
&&B_{i-1}\ar[d] \ar@{=}[r] & B_{i-1} \ar[d] \\
    & & 0& \hspace{0.15cm}0.  &}$
\end{center}For any $i$, if $\mathrm{Ext}^j_R(R/\mathfrak{a},C_i)$ is finitely generated for $j\leq d-2$, then $\mathrm{Ext}^j_R(R/\mathfrak{a},B_i)$ is finitely generated for $j\leq d-1$ by the second row. Hence $B_i$ is $\mathfrak{a}$-cofinite by Lemmas \ref{lem:3.42'} and \ref{lem:5.1} and \cite[Corollary 2.2.13]{F}, and then $\mathrm{H}_i(X)$ is $\mathfrak{a}$-cofinite. Conversely, if each $H_i$  is $\mathfrak{a}$-cofinite, then, from the associated long exact sequences induces by the second column and the first row, each $B_i$ is $\mathfrak{a}$-cofinite. Thus each $C_i$ is $\mathfrak{a}$-cofinite, as required.
\end{proof}

\bigskip
\section{\bf Cofiniteness of local cohomology modules}
The task of this section is to investigate cofiniteness of local cohomology $\mathrm{H}^{i}_\mathfrak{a}(X)$ for $X\in \mathrm{D}_\sqsubset(R)$ in the cases $\mathrm{dim}R=d\geq1$, or $\mathrm{dim}R/\mathfrak{a}=d-1$, or $\mathrm{ara}(\mathfrak{a})=d-1$.

\begin{prop}\label{lem:2.8}{\it{Let $\mathfrak{a}$ be a proper ideal of $R$ and $t$ an integer, and let $X\in\mathrm{D}_\sqsubset(R)$ such that $\mathrm{RHom}_R(R/\mathfrak{a},X)\in \mathrm{D}^\mathrm{f}(R)$.

$(1)$ If $\mathrm{H}^{i}_\mathfrak{a}(X)\in FD_{\leq0}$ or $\mathrm{dim}_R\mathrm{H}^{i}_\mathfrak{a}(X)\leq 0$ for $i<t$ or $\mathrm{dim}R\leq1$, then $\mathrm{H}^{i}_\mathfrak{a}(X)$ is $\mathfrak{a}$-cofinite for $i<t$ and $\mathrm{Hom}_R(R/\mathfrak{a},\mathrm{H}^{t}_\mathfrak{a}(X))$ is finitely generated.

$(2)$ If $\mathrm{H}^{i}_\mathfrak{a}(X)\in FD_{\leq1}$ or $\mathrm{dim}_R\mathrm{H}^{i}_\mathfrak{a}(X)\leq 1$ for $i<t$ or $\mathrm{dim}R=2$, then $\mathrm{H}^{i}_\mathfrak{a}(X)$ is $\mathfrak{a}$-cofinite for $i<t$ and $\mathrm{Hom}_R(R/\mathfrak{a},\mathrm{H}^{t}_\mathfrak{a}(X))$ and $\mathrm{Ext}^1_R(R/\mathfrak{a},\mathrm{H}^{t}_\mathfrak{a}(X))$ are finitely generated.

$(3)$ If either $\mathrm{dim}R=3$, or $\mathrm{dim}R/\mathfrak{a}=2$, or $\mathrm{ara}(\mathfrak{a})=2$, then $\mathrm{H}^{i}_\mathfrak{a}(X)$ is $\mathfrak{a}$-cofinite for $i\leq t-1$ if and only if $\mathrm{Hom}_R(R/\mathfrak{a},\mathrm{H}^{i}_\mathfrak{a}(X))$ is finitely generated for $i\leq t$.

$(4)$ If $R$ is local so that either $\mathrm{dim}R=d\geq4$, or $\mathrm{dim}R/\mathfrak{a}=d-1$, or $\mathrm{ara}(\mathfrak{a})=d-1$, then $\mathrm{H}^{i}_\mathfrak{a}(X)$ is $\mathfrak{a}$-cofinite for $i\leq t-1$ if and only if $\mathrm{Hom}_R(R/\mathfrak{a},\mathrm{H}^{i+d-3}_\mathfrak{a}(X)),\cdots,\mathrm{Ext}^{d-3}_R(R/\mathfrak{a},\mathrm{H}^{i}_\mathfrak{a}(X))$ are finitely generated for $i\leq t$.}}
\end{prop}
\begin{proof} We just prove (1) since (2) and (3) and (4) are similar.

 We may assume that $t\geq-\mathrm{sup}X$. Set $s=0$ and $\mathfrak{b}=\mathfrak{a}$ in Lemma \ref{lem:0.6}, it is enough to show that  $\mathrm{H}^{i}_\mathfrak{a}(X)$ is $\mathfrak{a}$-cofinite
for all $i<t$. We prove by induction on $t$. The case
$t=-\mathrm{sup}X$ is obvious. Suppose $t>-\mathrm{sup}X$ and the result has been proved for smaller values of
$t$. Then $\mathrm{H}^i_\mathfrak{a}(X)$ is $\mathfrak{a}$-cofinite for $i<t-1$ by the induction. Hence
Lemma \ref{lem:0.6} implies that $\mathrm{Hom}_R(R/\mathfrak{a},\mathrm{H}^{t-1}_\mathfrak{a}(X))$ is finitely generated, and so
 $\mathrm{H}^{t-1}_\mathfrak{a}(X)$ is $\mathfrak{a}$-cofinite by Lemma \ref{lem:6.6}. Now the assertion follows from Lemma \ref{lem:0.6}.
\end{proof}

\begin{rem}\label{rem:2.5} \rm
(1) The above proposition is a nice generalization of \cite[Theorem 2.5]{BN0} and \cite[Theorem 2.6]{BN}.

(2) By Proposition \ref{lem:2.8}(1) and (2), one can obtain that the set $\mathrm{Ass}_{R}\mathrm{H}_{\mathfrak{a}}^{t}(X)$ is finite, which generalizes a result due to Brodmann and Faghani \cite[Theorem 2.2]{BF}.
\end{rem}

\begin{cor}\label{lem:2.48}{\it{Let $\mathfrak{a}\subseteq\mathfrak{b}$ be proper ideals of $R$ and $X\in\mathrm{D}_\sqsubset(R)$ such that $\mathrm{RHom}_R(R/\mathfrak{a},X)\in \mathrm{D}^\mathrm{f}(R)$.

$(1)$ If either $\mathrm{dim}R/\mathfrak{a}\leq1$, or $\mathrm{dim}R\leq2$, or $\mathrm{ara}(\mathfrak{a})\leq1$, then
$\mathrm{H}^i_\mathfrak{a}(X)$ are $\mathfrak{a}$-cofinite for all $i\in\mathbb{Z}$. In these cases, the set $\mathrm{Ass}_R\mathrm{H}^i_\mathfrak{a}(X)$, the Bass number $\mu^{j}_{R}(\mathfrak{p},\mathrm{H}^i_\mathfrak{a}(X))$ and the Betti number $\beta^{j}_{R}(\mathfrak{p},\mathrm{H}^i_\mathfrak{a}(X))$ are finite for all $i,j\in\mathbb{Z}$.

$(2)$ If $\mathrm{dim}R/\mathfrak{b}\leq1$, then $\mathrm{H}^i_\mathfrak{b}(X)$ are $\mathfrak{b}$-cofinite for all $i\in\mathbb{Z}$.

$(3)$ If either $\mathrm{dim}R/\mathfrak{a}=2$ or $\mathrm{dim}R=3$ or $\mathrm{ara}(\mathfrak{a})=2$ and either $\mathrm{H}^{2i}_\mathfrak{a}(X)$ is $\mathfrak{a}$-cofinite or $\mathrm{H}^{2i+1}_\mathfrak{a}(X)$ is $\mathfrak{a}$-cofinite, then  $\mathrm{H}^{i}_\mathfrak{a}(X)$ is $\mathfrak{a}$-cofinite for every $i\in\mathbb{Z}$.}}
\end{cor}
\begin{proof} (1) This follows from Proposition \ref{lem:2.8}(2) and \cite[Corollary 2.4]{NS}.

(2) By \cite[Proposition 7.2]{WW}, $\mathrm{RHom}_R(R/\mathfrak{b},X)\in \mathrm{D}^\mathrm{f}(R)$. The statement follows from (1).

(3) The assertion follows from the proposition \ref{lem:2.8}(3) and using an induction on $i$.
\end{proof}

The following theorem gives some conditions for the correctness of the isomorphism $\mathrm{Ext}^{s+t}_R(R/\mathfrak{a},X)\cong\mathrm{Ext}^{s}_R(R/\mathfrak{a},\mathrm{H}^{t}_\mathfrak{a}(X))$, which was proved by Behrouzian and Aghapournahr when $X$ is an
$R$-module (see \cite[Theorem 4.4]{BA}).

\begin{thm}\label{lem:3.41}{\it{Let $\mathfrak{a},\mathfrak{b}$ be two ideals of $R$ with $\mathfrak{b}\subseteq\mathfrak{a}$, $X\in\mathrm{D}_\sqsubset(R)$ and $s\geq0,t\geq-\mathrm{sup}X$ such that

$(1)$ $\mathrm{Ext}^{s+t-i}_R(R/\mathfrak{a},\mathrm{H}^{i}_\mathfrak{b}(X))=0$ for all $-\mathrm{sup}X\leq i< t$ or $t+1\leq i\leq s+t$;

$(2)$ $\mathrm{Ext}^{s+1+i}_R(R/\mathfrak{a},\mathrm{H}^{t-i}_\mathfrak{b}(X))=0$ for all $0\leq i\leq t+\mathrm{sup}X$;

$(3)$ $\mathrm{Ext}^{s-1-i}_R(R/\mathfrak{a},\mathrm{H}^{t+i}_\mathfrak{b}(X))=0$ for all $0\leq i\leq s-1$.\\
Then $\mathrm{Ext}^{s+t}_R(R/\mathfrak{a},X)\cong\mathrm{Ext}^{s}_R(R/\mathfrak{a},\mathrm{H}^{t}_\mathfrak{b}(X))$.}}
\end{thm}
\begin{proof} Consider the first spectral sequence in Lemma \ref{lem:0.2}. There is a finite filtration\begin{center}
$0=U^{-s-t-\mathrm{sup}X-1}\subseteq U^{-s-t-\mathrm{sup}X}\subseteq\cdots \subseteq U^{0}=\mathrm{Ext}^{s+t}_R(R/\mathfrak{a},X)$,
\end{center}such that $U^{p}/U^{p-1}\cong E^\infty_{p,-s-t-p}$ for $-s-t\leq p+\mathrm{sup}X$.
Let $r\geq 2$. Consider the differential
\begin{center}$E^r_{-s+r,-t-r+1}\xrightarrow{d^r_{-s+r,-t-r+1}}E^r_{-s,-t}
\xrightarrow{d^r_{-s,-t}}E^r_{-s-r,-t+r-1}$.
\end{center}By conditions (2) and (3), we have $E^r_{-s+r,-t-r+1}=0=E^r_{-s-r,-t+r-1}$ for $r\geq 2$. By condition (1)
we have $0=U^{-s-t-\mathrm{sup}X-1}=\cdots=U^{-s-1}$ and $U^{-s}=\cdots=U^{0}=\mathrm{Ext}^{s+t}_R(R/\mathfrak{a},X)$.
So
\begin{center}$\mathrm{Ext}^{s}_R(R/\mathfrak{a},\mathrm{H}^{t}_\mathfrak{b}(X))=E^{2}_{-s,-t}\cong E^{\infty}_{-s,-t}\cong U^{-s}=\mathrm{Ext}^{s+t}_R(R/\mathfrak{a},X)$.\end{center}
We get the isomorphism we seek.
\end{proof}

The following corollary is a generalization of \cite[Corollary 4.5]{BA}.

\begin{cor}\label{lem:3.42}{\it{Let $\mathfrak{p}\in\mathrm{V}(\mathfrak{a})$, $X$ be in $\mathrm{D}_\sqsubset(R)$ and $s\geq0,t\geq-\mathrm{sup}X$. Suppose that
$\mathrm{Ext}^{s+t-i}_R(R/\mathfrak{p},\mathrm{H}^{i}_\mathfrak{a}(X))=0$ for $-\mathrm{sup}X\leq i< t$ or $t+1\leq i\leq s+t$;
$\mathrm{Ext}^{s+1+i}_R(R/\mathfrak{p},\mathrm{H}^{t-i}_\mathfrak{a}(X))=0$ for $0\leq i\leq t+\mathrm{sup}X$;
$\mathrm{Ext}^{s-1-i}_R(R/\mathfrak{p},\mathrm{H}^{t+i}_\mathfrak{a}(X))=0$ for $0\leq i\leq s-1$.
Then \begin{center}$\mu^{s+t}_{R}(\mathfrak{p},X)=\mu^{s}_R(\mathfrak{p},\mathrm{H}^{t}_\mathfrak{a}(X))$.\end{center}}}
\end{cor}
\begin{proof} By Theorem \ref{lem:3.41}, we have $\mathrm{Ext}^{s+t}_{R_\mathfrak{p}}(R_\mathfrak{p}/\mathfrak{p}R_\mathfrak{p},X_\mathfrak{p})\cong
\mathrm{Ext}^{s}_{R_\mathfrak{p}}(R_\mathfrak{p}/\mathfrak{p}R_\mathfrak{p},\mathrm{H}^{t}_\mathfrak{a}(X)_\mathfrak{p})$, which yields the desired equality.
\end{proof}

\begin{cor}\label{lem:3.43}{\it{Let $\mathfrak{a}$ be a proper ideal of $R$, $N$ a finitely generated $a$-torsion $R$-module, $X\in\mathrm{D}_\sqsubset(R)$ and $s\geq0,t\geq-\mathrm{sup}X$. Suppose that
$\mathrm{Ext}^{s+t-i}_R(N,\mathrm{H}^{i}_\mathfrak{a}(X))=0$ for $-\mathrm{sup}X\leq i< t$ or $t+1\leq i\leq s+t$;
$\mathrm{Ext}^{s+1+i}_R(N,\mathrm{H}^{t-i}_\mathfrak{a}(X))=0$ for $0\leq i\leq t+\mathrm{sup}X$;
$\mathrm{Ext}^{s-1-i}_R(N,\mathrm{H}^{t+i}_\mathfrak{a}(X))=0$ for $0\leq i\leq s-1$.
Then \begin{center}$\mathrm{Ext}^{s+t}_R(N,X)\cong\mathrm{Ext}^{s}_R(N,\mathrm{H}^{t}_\mathfrak{a}(X))$.\end{center}}}
\end{cor}
\begin{proof} For $N$, there is a finite filtration $0=N_0\subseteq N_1\subseteq\cdots\subseteq N_r=N$ such that $N_i/N_{i-1}\cong R/\mathfrak{a}$. A successive use of the short exact sequence
\begin{center}$0\rightarrow N_{i-1}\rightarrow N_i\rightarrow N_i
/N_{i-1}\rightarrow0$\end{center}
yields the desired isomorphism.
\end{proof}

\begin{cor}\label{lem:3.44}{\it{Let $\mathfrak{a}$ be a proper ideal of $R$, $Y\in\mathrm{D}^\mathrm{f}_\square(R)$ with $\mathrm{Supp}_RY\subseteq\mathrm{V}(\mathfrak{a})$, $X\in\mathrm{D}_\sqsubset(R)$ and $s\geq-\mathrm{inf}Y,t\geq-\mathrm{sup}X$. Suppose that
$\mathrm{Ext}^{s+t-i}_R(Y,\mathrm{H}^{i}_\mathfrak{a}(X))=0$ for $-\mathrm{sup}X\leq i< t$ or $t+1\leq i\leq s+t+\mathrm{inf}Y$;
$\mathrm{Ext}^{s+1+i}_R(Y,\mathrm{H}^{t-i}_\mathfrak{a}(X))=0$ for $0\leq i\leq t+\mathrm{sup}X$;
$\mathrm{Ext}^{s-1-i}_R(Y,\mathrm{H}^{t+i}_\mathfrak{a}(X))=0$ for $0\leq i\leq s+\mathrm{inf}Y-1$.
Then \begin{center}$\mathrm{Ext}^{s+t+\mathrm{inf}Y}_R(Y,X)\cong\mathrm{Ext}^{s+\mathrm{inf}Y}_R(Y,\mathrm{H}^{t}_\mathfrak{a}(X))$.\end{center}}}
\end{cor}
\begin{proof} We use induction on $\mathrm{sup}Y-\mathrm{inf}Y$.
 If $\mathrm{inf}Y=\mathrm{sup}Y=r$, then $Y\simeq\Sigma^r\mathrm{H}_r(Y)$ and $\mathrm{Ext}^{s+t+\mathrm{inf}Y}_R(Y,X)\cong\mathrm{Ext}^{s+t}_R(\mathrm{H}_r(Y),X)
 \cong\mathrm{Ext}^{s+\mathrm{inf}Y}_R(Y,\mathrm{H}^{t}_\mathfrak{a}(X))$ for all $i\leq s$ by Corollary \ref{lem:3.43}. Now assume that $\mathrm{sup}Y-\mathrm{inf}Y>0$. Set $Y'=0\rightarrow Y_{\mathrm{sup}Y}/\mathrm{Ker}d_{\mathrm{sup}Y}\xrightarrow{\bar{d}_{\mathrm{sup}Y}}Y_{\mathrm{sup}Y-1}\xrightarrow{d_{\mathrm{sup}Y-1}}\cdots$. One obtain an exact
 triangle $\Sigma^{\mathrm{sup}Y}\mathrm{H}_{\mathrm{sup}Y}(Y)\rightarrow Y\rightarrow Y'\rightsquigarrow$ in $\mathrm{D}(R)$, which induces the following exact sequence \begin{center}$\mathrm{Ext}^{s+t+\mathrm{inf}Y}_R(Y',X)\rightarrow\mathrm{Ext}^{s+t+\mathrm{inf}Y}_R(Y,X)
\rightarrow\mathrm{Ext}^{s+t+\mathrm{inf}Y}_R(\Sigma^{\mathrm{sup}Y}\mathrm{H}_{\mathrm{sup}Y}(Y),X)$.\end{center} Therefore, $\mathrm{Ext}^{s+t+\mathrm{inf}Y}_R(Y,X)\cong\mathrm{Ext}^{s+\mathrm{inf}Y}_R(Y,\mathrm{H}^{t}_\mathfrak{a}(X))$ by the induction.
\end{proof}

\bigskip \centerline {\bf Acknowledgments} This research was partially supported by National Natural Science Foundation of China (11761060,11901463).

\end{document}